\RequirePackage[reqno]{amsmath}
\documentclass[preprint]{amsart}

\usepackage{amssymb}
\usepackage{empheq}
\usepackage{amsthm}
\usepackage{enumerate}
\usepackage[top=1in, left=1in, right=0.9in, bottom=1in]{geometry}

\newtheorem{theorem}{Theorem}[section]
\newtheorem{lemma}[theorem]{Lemma}
\newtheorem{definition}[theorem]{Definition}
\newtheorem{proposition}[theorem]{Proposition}

\def\R{{\mathbb R}}

\def\S{{\mathbb S}}

\newcommand{\cF}{\mathcal{F}}

\newcommand{\ds}{\displaystyle{}}

\newcommand{\overbar}[1]{\mkern 1.5mu\overline{\mkern-1.5mu#1\mkern-1.5mu}\mkern 1.5mu}

\newcommand{\cp}{\operatorname{cap}}
\newcommand{\supp}{\operatorname{supp}}
\newcommand{\dist}{\operatorname{dist}}

\newcommand{\Beta}{{\rm B}}

\begin{document}

\title{Minimum Energy Problem on the Hypersphere}

\author[Mykhailo Bilogliadov]{Mykhailo Bilogliadov}
\email{mykhail@okstate.edu}
\address{Department of Mathematics, Oklahoma State University, Stillwater, OK 74078, U.S.A.}

%\date{March 9, 2016}

\begin{abstract}
We consider the minimum energy problem on the unit sphere $\S^{d-1}$ in the Euclidean space $\R^d$, $d\geq 3$, in the presence of an external field $Q$, where the charges are assumed to interact according to  Newtonian 
potential $1/r^{d-2}$, with $r$ denoting the Euclidean distance. We solve the problem by finding the support of the 
extremal measure, and obtaining an explicit expression for the density of the extremal measure. We then apply our results to an external field generated by a point charge of positive magnitude, placed at the North Pole of the sphere, and to a quadratic external field.


\end{abstract}

\maketitle

\medskip
\medskip
\noindent

\section{Introduction}

Let $\S^{d-1}:=\{x\in\R^d: |x|=1\}$ be the unit sphere in $\R^d$, $d\geq 3$, where $|\cdot|$ is the Euclidean distance. Given a compact set $E\subset\S^{d-1}$, consider the class $\mathcal{M}(E)$ of unit positive Borel measures supported on $E$. The Newtonian potential and Newtonian energy of a measure $\mu\in\mathcal{M}(E)$ are defined respectively as 
$$U^{\mu}(x):=\int \frac {1} {|x-y|^{d-2}} \, d\mu(y), \quad I(\mu):=\iint \frac {1} {|x-y|^{d-2}} \, d\mu(x)d\mu(y).$$
Let 
$$W(E):=\inf\{I(\mu): \mu\in\mathcal{M}(E)\}.$$
Define the Newtonian capacity of $E$ as $\cp(E):=1/W(E)$. We say that a property holds quasi-everywhere (q.e.), if the exceptional set has a Newtonian capacity zero. When $\cp(E)>0$, there is a unique $\mu_{E}$ such that $I(\mu_E)=W(E)$. Such $\mu_E$ is called the Newtonian equilibrium measure for $E$.

An external field is defined as a non-negative lower-semicontinuous function $Q: E \rightarrow [0,\infty]$, such that $Q(x)<\infty$ on a set of positive Lebesgue surface measure. The weighted energy associated with $Q(x)$ is then defined by $$I_Q(\mu):=I(\mu)+2\int Q(x)d\mu(x).$$

\noindent The energy problem on $\S^{d-1}$ in the presence of the external field $Q(x)$ refers to the minimal quantity 
$$V_Q:=\inf\{I_Q(\mu): \mu\in\mathcal{M}(E)\}.$$ 
A measure $\mu_Q \in\mathcal{M}(E)$ such that $I_Q(\mu_Q)=V_Q$ is called an extremal (or positive Newtonian equilibrium) measure associated with $Q(x)$.

The potential $U^{\mu_Q}$ of the measure $\mu_Q$ satisfies the Gauss variational inequalities
\begin{empheq}{align}\label{var1}
U^{\mu_Q}(x)+Q(x)&\geq F_Q \quad \text{q.e. on}\ \ E,\\\label{var2}
U^{\mu_Q}(x)+Q(x)& \leq F_Q \quad \text{for all}\ \ x\in S_Q,
\end{empheq}
where $F_Q:=V_Q-\ds{ \int Q(x)\,d\mu_Q(x)}$, and $S_Q:=\supp{\mu_Q}$ (see Theorem 10.3 in \cite{mizuta} or Proposition 3 in \cite{bds2}, and also book \cite{bhs}). We remark that for continuous external fields, the inequality in $(\ref{var1})$ holds everywhere, which implies that equality holds in $(\ref{var2})$.

The minimum energy problems with external fields on the sphere were a subject of investigation by 
the group of Brauchart, Dragnev and Saff \cite{bds1}--\cite{bds3}, \cite{ddss}, see also \cite{bilo}. 
For a comprehensive treatment of the subject one should refer to a forthcoming book \cite{bhs}. 

The main aim of this paper is to provide a solution to the weighted energy problem on the sphere 
$\S^{d-1}\subset\R^d$, $d\geq 3$, immersed in a general external field, possessing rotational symmetry with respect to the polar 
axis, when support is a spherical cap. It is assumed 
that the charges interact according to the classical Newtonian potential $1/r^{d-2}$, where $r$ denotes the Euclidean distance. We obtain an equation that describes the support of the 
extremal measure, while also giving an explicit expression for the equilibrium density, thus extending an approach previously developed in \cite{bilo} for the case $d=3$. As an application of our results, we consider an external field produced by a point charge of positive magnitude, placed at the North Pole of the sphere, and explicitly compute support and density of the corresponding extremal measure. As another application of the developed theory, we consider the case when the sphere is immersed into a quadratic external field, and also explicitly find support and density of the extremal measure.

We remark that the density of the extremal measure when the external field is generated by a point charge of positive magnitude, located at the North Pole of the sphere, for the case of general Riesz potentials, was first found in \cite{ddss}. We also note that the results of \cite{bds1} can be extended to cover such a case as well, again for the general Riesz potentials. In work \cite{bilo} this case when $d=3$ and the particles were assumed to interact via Newtonian potential, was investigated using an approach different from the one employed in \cite{ddss} and \cite{bds1}. 

We introduce hyperspherical polar coordinates $r, \theta_1, \theta_2,\ldots, \theta_{d-2}, \varphi$, defined by
\begin{align*}
x_1 & = r\cos\theta_1,\\
x_2 & = r\sin\theta_1\cos\theta_2,\\
x_3 & = r\sin\theta_1\sin\theta_2\cos\theta_3,\\
& \vdots \\
x_{d-2} & = r\sin\theta_1\sin\theta_2\ldots\sin\theta_{d-3} \cos\theta_{d-2},\\
x_{d-1} & = r\sin\theta_1\sin\theta_2\ldots\sin\theta_{d-2} \cos\varphi,\\
x_{d} & = r\sin\theta_1\sin\theta_2\ldots\sin\theta_{d-2} \sin\varphi,
\end{align*}
where $0\leq r$, $0\leq\theta_j\leq \pi$, $j=1,2,\ldots,d-2$, and $0\leq\varphi \leq 2\pi$, see \cite{be2}. The surface area element of the unit sphere $\S^{d-1}$, written in hyperspherical coordinates, is given by
\[
d\sigma_d = \sin^{d-2} \theta_1 \, \sin^{d-3} \theta_2 \ldots \sin\theta_{d-2} \, d\theta_1\, d\theta_2 \ldots d\theta_{d-2}\, d\varphi.
\]
The total surface area of the sphere $\S^{d-1}$ is given by 
\[
\omega_d=\frac{2\pi^{d/2}}{\Gamma(d/2)}.
\]
\noindent A spherical cap on the sphere $\S^{d-1}$, centered at the North Pole, is defined via an angle $\alpha$, $0 < \alpha \leq \pi$, as
\begin{equation*} 
C_{N,\alpha}:= \{ (r,\theta_1,\ldots,\theta_{d-2},\varphi): r=1,\ \ 0\leq \theta_1 \leq \alpha, \ \ 0\leq \theta_j \leq \pi, j=2,\ldots,d-2, \ \  0\leq \theta \leq 2\pi  \}.
\end{equation*}
Similarly, a spherical cap centered at the South Pole, is defined in terms of an angle $\alpha$, $0 < \alpha \leq \pi$, as
\begin{equation*} 
C_{S,\alpha}:= \{ (r,\theta_1,\ldots,\theta_{d-2},\varphi): r=1,\ \ \alpha \leq \theta_1 \leq \pi, \ \ 0\leq \theta_j \leq \pi, j=2,\ldots,d-2, \ \  0\leq \theta \leq 2\pi  \}.
\end{equation*}

\noindent In what follows, we will need to use certain special functions, for which we fix the notation here. The incomplete Beta function $\Beta(z; a,b)$ is defined as
\begin{equation}\label{betafdef}
\Beta(z;a,b): = \int_0^z t^{a-1} (1-t)^{b-1}\, dt.
\end{equation}
The Gauss hypergeometric function $_2 F_1(a,b;c,z)$ is defined via series
\begin{equation}\label{gausshyperdef}
_2 F_1(a,b;c,z): = \sum_{n=0}^\infty \frac {(a)_n\, (b)_n} {(c)_n} \, \frac {z^n} {n!}, \quad |z|<1,
\end{equation}
where $(a)_0: = 1$ and $(a)_n: = a (a+1) \ldots (a+n-1)$ for $n\geq 1$ is the Pochhammer symbol.

We begin by recording sufficient conditions on an external field $Q$, that guarantee that the extremal support 
$S_Q$ of the equilibrium measure $\mu_Q$ is a spherical cap $C_{S,\alpha}$, centered at the South Pole.
The following proposition is a consequence of a more general statement, proved in \cite{bds1}.
\begin{proposition}\label{scssp}
Let a non-negative external field $Q$ be rotationally invariant with respect to rotations about the polar axis $x_1$, that is $Q(x)=Q(x_1)$, where 
$x=(x_1,x_2,\ldots,x_d)\in\S^{d-1}$. Suppose that $Q(x_1)$ is an increasing convex function on $[-1,1]$. Then the support of the 
extremal measure $\mu_Q$ is a spherical cap centered at the South Pole, that is $S_Q=C_{S,\alpha}$.
\end{proposition}

The following result is an important step towards the recovery of the equilibrium measure.
\begin{theorem}\label{theo1}
Suppose that an external field $Q$ is rotationally invariant with respect to rotations about the polar axis, and is such that $Q\in C^2(N)$, where $N$ is an open neighborhood of $S_Q=C_{S,\alpha}$ on the sphere $\S^{d-1}$. 
Then the equilibrium measure $\mu_Q$ is absolutely continuous with respect to the Lebesgue 
surface measure, with a locally bounded density, that is $d\mu_Q =  f(\theta_1)\,d\sigma_d$, where $f\in L^\infty([\alpha,\pi])$.
\end{theorem}

The support $S_Q$ is a main ingredient in determining the equilibrium measure $\mu_Q$ itself. Indeed, if 
$S_Q$ is known, then the equilibrium measure $\mu_Q$ can be recovered by solving the singular integral equation
\begin{equation}\label{ieec}
\int \frac{1}{|x-y|^{d-2}}\,d\mu(y)+Q(x)=F_Q, \quad x\in S_Q,
\end{equation}
where $F_Q$ is a constant (see (\ref{var2})).

We solve this equation and obtain the following two theorems, that describe explicitly the 
equilibrium density when support $S_Q$ is either a spherical cap $C_{N,\alpha}$, centered at the North Pole, or $C_{S,\alpha}$, a spherical
cap centered at the South Pole. This extends the corresponding results stated in Theorem 2 and Theorem 3 in \cite{bilo}, for the case $d=3$.
\begin{theorem}\label{theo2}
Suppose that an external field $Q$ is rotationally invariant with respect to rotations about the polar axis, and is such that $Q\in C^2(N)$, where $N$ is an open neighborhood of $S_Q$ in $\S^{d-1}$. Assume that $S_Q$ is a spherical
cap $C_{N,\alpha}$ centered at the North Pole, with $0 < \alpha \leq \pi$. Let
\begin{equation}\label{auxfNP}
F(\eta) =  \frac  { \Gamma((d-2)/2)} {2\, \pi^{(d+2)/2}}  \, \frac {1} {\sin \eta} \,\sec^{d-3}\bigg(\frac{\eta}{2}\bigg) \, \frac {d} {d\eta} \int_\eta^\alpha \frac {g(\zeta) \sin \zeta \, d\zeta} {\sqrt{\cos \eta - \cos \zeta}}, \quad 0  \leq \eta \leq \alpha,
\end{equation}
where
\begin{equation}\label{auxsNP}
g(\zeta) = \cot^{d-3}\bigg(\frac{\zeta}{2}\bigg)\, \frac {d} {d\zeta} \int_0^\zeta \frac {Q(\theta)\,  \sin^{d-3}(\theta/2) \sin\theta \, d\theta} {\sqrt{\cos\theta - \cos \zeta}}, \quad 0 \leq \zeta \leq \alpha.
\end{equation}
Then the density $f$ of the equilibrium measure $\mu_Q$ is given by
\begin{equation}\label{equildensNP}
f(\eta)  =   C_Q  \bigg(\frac{1+\cos\alpha}{1+\cos\eta} \bigg)^{\frac{d-1}{2}}  \bigg(\frac{1+\cos\alpha}{\cos\eta-\cos\alpha} \bigg)^{\frac{1}{2}} \, {}_2 F_1\bigg (1, \frac {d-1} {2};\frac {1} {2};  \frac{\cos\eta-\cos\alpha}{1+\cos\eta} \bigg) + F(\eta), \quad 0 \leq \eta \leq \alpha.
\end{equation}
The constant $C_Q$ is uniquely defined by
\begin{equation}\label{constCQNP}
C_Q  =  \frac {\Gamma(d/2-1)} {2^{d-2}\, \sqrt{\pi}\, \Gamma((d-1)/2) }  \, \bigg(\Beta\bigg(\sin^2\bigg(\frac {\alpha} {2} \bigg); \frac{d-2}{2}, \frac{d}{2} \bigg) \bigg)^{-1} \,  \bigg\{  \frac{\Gamma((d-1)/2)}{2\pi^{(d-1)/2}} -  \int_0^\alpha F(\eta) \, \sin^{d-2}\eta \, d\eta \bigg\}.
\end{equation}
\end{theorem}
\noindent An analogous statement for the support being a spherical cap $C_{S,\alpha}$ centered at the South Pole, 
is of the following nature.
\begin{theorem}\label{theo3}
Suppose that an external field $Q$ is rotationally invariant with respect to rotations about the polar axis, and is such that $Q\in C^2(N)$, where $N$ is an open neighborhood of $S_Q$ in $\S^{d-1}$. Assume that $S_Q$ is a spherical
cap $C_{S,\alpha}$, centered at the South Pole, with $0 < \alpha \leq \pi$. Let
\begin{equation}\label{auxfSP}
F(\eta) =  \frac  { \Gamma((d-2)/2)} {2\, \pi^{(d+2)/2}}  \, \frac {1} {\sin\eta } \,\csc^{d-3}\bigg(\frac{\eta}{2}\bigg) \, \frac {d} {d\eta} \int_\alpha^{\eta}\frac {g(\zeta) \sin \zeta \, d\zeta} {\sqrt{\cos\zeta - \cos \eta}}, \quad \alpha \leq \eta \leq \pi,
\end{equation}
where
\begin{equation}\label{auxsSP}
g(\zeta) = \tan^{d-3}\bigg(\frac{\zeta}{2}\bigg)\, \frac {d} {d\zeta} \int_\zeta^\pi \frac {Q(\theta)\,  \cos^{d-3}(\theta/2) \sin\theta \, d\theta} {\sqrt{\cos \zeta-\cos\theta}}, \quad \alpha \leq \zeta \leq \pi.
\end{equation}
Then the density $f$ of the equilibrium measure $\mu_Q$ is given by
\begin{equation}\label{equildensSP}
f(\eta)  =   C_Q  \bigg(\frac{1-\cos\alpha}{1-\cos\eta} \bigg)^{\frac{d-1}{2}}  \bigg(\frac{1-\cos\alpha}{\cos\alpha-\cos\eta} \bigg)^{\frac{1}{2}} \, {}_2 F_1\bigg (1, \frac {d-1} {2};\frac {1} {2};  \frac{\cos\alpha-\cos\eta}{1-\cos\eta} \bigg) + F(\eta), \quad \alpha \leq \eta \leq \pi.
\end{equation}
The constant $C_Q$ is uniquely determined by
\begin{equation}\label{constCQSP}
C_Q  = \frac {\Gamma(d/2-1)} {2^{d-2}\, \sqrt{\pi} \, \Gamma((d-1)/2) } \, \bigg(\Beta\bigg(\cos^2\bigg(\frac {\alpha} {2} \bigg); \frac{d-2}{2}, \frac{d}{2} \bigg) \bigg)^{-1} \,  \bigg\{  \frac{\Gamma((d-1)/2)}{2\pi^{(d-1)/2}}  - \int_\alpha^\pi F(\eta) \, \sin^{d-2}\eta \, d\eta \bigg\}.
\end{equation}
\end{theorem}

% ----------------- Section ' Applications ' ---------------------------------- Section ' Applications ' ---------------------------------- Section ' Applications ' ---------------------------------- Section ' Applications ' -----------------

\section{Applications}
Following the approach developed in \cite{bilo}, we first consider the case of no external field, when the support is a spherical 
cap centered at the South Pole, that is $S_Q=C_{S,\alpha}$, $Q=0$. The equilibrium measure for 
the spherical cap centered at the South Pole $C_{S,\alpha}$, for the case of general Riesz potentials, was first 
found in \cite{ddss} (see also \cite{bds1}).

\begin{proposition}
The density of the equilibrium measure of a spherical cap $C_{S,\alpha}$ with no external field is
\begin{empheq}{align}\label{equildensSP-no-field}
f(\eta)  =    \frac {\Gamma(d/2-1)} {2^{d-1}\, \pi^{d/2}} \,  \bigg(\Beta\bigg(\cos^2\bigg(\frac {\alpha} {2} \bigg); \frac{d-2}{2}, & \frac{d}{2} \bigg) \bigg)^{-1} \,  \bigg(\frac{1-\cos\alpha}{1-\cos\eta} \bigg)^{\frac{d-1}{2}}  \bigg(\frac{1-\cos\alpha}{\cos\alpha-\cos\eta} \bigg)^{\frac{1}{2}} \\ \nonumber
& \times {}_2 F_1\bigg (1, \frac {d-1} {2};\frac {1} {2};  \frac{\cos\alpha-\cos\eta}{1-\cos\eta} \bigg), \quad \alpha \leq \eta \leq \pi.
\end{empheq}
The capacity of $C_{S,\alpha}$ is given by
\begin{equation}\label{capSP}
\cp(C_{S,\alpha}) = \frac {2^{d-2} \, \Gamma((d-1)/2)} {\sqrt{\pi} \, \Gamma(d/2-1)}  \,  \Beta\bigg(\cos^2\bigg(\frac {\alpha} {2} \bigg); \frac{d-2}{2}, \frac{d}{2} \bigg).
\end{equation}
\end{proposition}

Suppose now that the sphere $\S^{d-1}$ is immersed in an external field $Q$, that satisfies the conditions of 
Proposition \ref{scssp}. Then the support $S_Q$ of the weighted equilibrium measure $\mu_Q$ will be 
a spherical cap $C_{S,\alpha}$, centered at the South Pole. The angle $\alpha$, defining the extremal support $C_{S,\alpha}$, can be found 
via the Newtonian analog of the Mhaskar-Saff $\cF$-functional, which is defined as follows.
\begin{definition}\label{msfdef}
The  $\cF$-functional of a compact subset $E\subset \S^{d-1}$ of positive (Newtonian) capacity is defined as
\begin{equation}\label{msff}
\cF(E) := W(E) + \int Q(x)\,d\mu_E(x), 
\end{equation}
where $W(E)$ is the Newtonian energy of the compact $E$ and $\mu_E$ is the equilibrium measure (with no external field) on $E$.
\end{definition}
The main objective of introducing the $\cF$-functional is its following extremal property, proved in \cite{bds1} for the general Riesz potentials.
\begin{proposition}\label{msfmp}
Let $Q$ be an external field on $\S^{d-1}$. Then $\cF$-functional is minimized for $S_Q=\supp(\mu_Q)$. 
\end{proposition}
If $E=C_{S,\alpha}$, taking into account that $W(C_{S,\alpha})=1/\cp(C_{S,\alpha})$, and inserting $(\ref{capSP})$ and $(\ref{equildensSP-no-field})$ into $(\ref{msff})$, we find that $\cF$-functional is given by
\begin{empheq}{align}\label{msfunc}
 \cF(C_{S,\alpha}) & =  \frac {\sqrt{\pi}\,  \Gamma(d/2-1)} {2^{d-2}\,  \Gamma((d-1)/2)} \, \bigg(\Beta\bigg(\cos^2\bigg(\frac {\alpha} {2} \bigg); \frac{d-2}{2},  \frac{d}{2} \bigg) \bigg)^{-1}  \times \\ \nonumber 
\bigg\{ 1 & + \frac {1} {\pi} \int_\alpha^\pi Q(\eta)\,  \bigg(\frac{1-\cos\alpha}{1-\cos\eta} \bigg)^{\frac{d-1}{2}}  \bigg(\frac{1-\cos\alpha}{\cos\alpha-\cos\eta} \bigg)^{\frac{1}{2}} {}_2 F_1\bigg (1, \frac {d-1} {2};\frac {1} {2};  \frac{\cos\alpha-\cos\eta}{1-\cos\eta} \bigg) \sin^{d-2}\eta\, d\eta \bigg\}.
\end{empheq}

As a first applications of our results, we consider the situation when the sphere $\S^{d-1}$ is immersed in an external field generated by a positive point charge of magnitude $q$ placed at the North Pole of the sphere, namely
\begin{equation}\label{extfieldexam}
Q(x) = q\,(1-x_1)^{-(d-2)/2}, \quad q > 0, \quad x\in\S^{d-1}.
\end{equation}
Note that the extremal measure for such a field was first obtained in \cite{ddss}, for general Riesz potentials. For $d=3$ and when the charges are assumed to interact according to the Newtonian potential, the extremal density and Mhaskar-Saff functional, along with its critical points, were computed in \cite{bilo} (see Proposition 4 and Theorems 4 and 5 there). We also remark that it is possible to extend the results of \cite{bds1} to cover such a case as well.

It is clear that external field $Q$ in $(\ref{extfieldexam})$ is invariant with respect to the rotations about the polar axis. Also, $Q(x_1)$ is a non-negative increasing convex function on $[-1,1]$. From Proposition \ref{scssp} it then follows that the support of the corresponding equilibrium measure $\mu_Q$ will be a spherical cap $C_{S,\alpha}$, centered at the South Pole. 
The crucial step towards the recovery of the equilibrium measure for this external field is to determine the support of the equilibrium measure.  For that we first compute the Mhaskar-Saff $\cF$-functional, by inserting expression $(\ref{extfieldexam})$ for the external field $Q$ into $(\ref{msfunc})$.
\begin{theorem}\label{msfptch}
The $\cF$-functional for the spherical cap $C_{S,\alpha}$ when the external field is produced by a positive point charge of magnitude $q$, placed at the North Pole, is given by
\begin{empheq}{align}\label{msf-p-ch}
 \cF(C_{S,\alpha})  =  \frac {\sqrt{\pi}\,  \Gamma(d/2-1)} {2^{d-2}\,  \Gamma((d-1)/2)} \,   \bigg(\Beta\bigg(\cos^2\bigg(\frac {\alpha} {2} \bigg); & \frac{d-2}{2},  \frac{d}{2} \bigg) \bigg)^{-1}  \times \\ \nonumber
 & \bigg\{ 1 + \frac {q\,  2^{(d-2)/2} \, \Gamma((d-1)/2)} {\sqrt{\pi} \, \Gamma(d/2-1)} \, \Beta \bigg( \cos^2\bigg(\frac {\alpha} {2} \bigg) ; \frac{d-2}{2},\frac{1}{2}\bigg) \bigg\}.
\end{empheq}
\end{theorem}

\noindent Applying Theorem \ref{theo3}, we compute the density of the equilibrium measure, corresponding to this external field. 
\begin{theorem}\label{densexamtheo}
For the external field $Q$ given by $(\ref{extfieldexam})$, the support $S_Q$ is a spherical cap $C_{S,\alpha}$ centered at the South Pole, with $\alpha=\alpha_0\in(0,\pi)$. The 
angle $\alpha_0$ is defined as a unique solution of the equation
\begin{equation}\label{eqn-crit-pnts}
 \csc^{d-1}\bigg(\frac {\alpha} {2} \bigg) \, \Beta\bigg(\cos^2\bigg(\frac {\alpha} {2} \bigg);  \frac{d-2}{2},  \frac{d}{2} \bigg) - \Beta\bigg(\cos^2\bigg(\frac {\alpha} {2} \bigg);  \frac{d-2}{2},  \frac{1}{2} \bigg) = \frac  {\sqrt{\pi}\,  \Gamma(d/2-1)} {q\,2^{(d-2)/2}\, \Gamma((d-1)/2)}.
\end{equation}
Let 
\begin{empheq}{align}\label{CQ-val-theo}
C_Q  = \frac {\Gamma(d/2-1)} {2^{d-1}\, \pi^{d/2}} \, \bigg(\Beta\bigg(\cos^2\bigg(\frac{\alpha_0}{2}\bigg);  \frac{d-2}{2},   \frac{d}{2}\bigg) & \bigg)^{-1} \times \\ \nonumber
& \bigg\{ 1 + \frac {q \, 2^{(d-2)/2}\, \Gamma((d-1)/2)} {\sqrt{\pi}\, \Gamma(d/2-1)} \,  \Beta\bigg(\cos^2\bigg(\frac{\alpha_0}{2}\bigg); \frac{d-2}{2}, \frac{1}{2}\bigg)  \bigg\}.
\end{empheq}
The density of the equilibrium measure $\mu_Q$ is given by
\begin{empheq}{align}\label{densExamp}
f(\eta)  =   C_Q\,  \bigg(\frac{1-\cos\alpha_0}{1-\cos\eta} \bigg)^{(d-1)/2} \, & \sqrt{\frac{1-\cos\alpha_0}{\cos\alpha_0-\cos\eta}}  \, {}_2 F_1\bigg (1, \frac {d-1} {2};\frac {1} {2};  \frac{\cos\alpha_0-\cos\eta}{1-\cos\eta} \bigg) \\ \nonumber
& - \frac {q\, \Gamma((d-1)/2)}  {\sqrt{2}\,  \pi^{(d+1)/2}}  \, \frac {1} {(1-\cos\eta)^{(d-1)/2}} \, \sqrt{\frac{1-\cos\alpha_0}{\cos\alpha_0-\cos\eta}}, \quad \alpha_0 \leq \eta \leq \pi.
\end{empheq}
\end{theorem}

As a second application, we consider the case when the external field $Q$ is a quadratic polynomial of the form
\begin{equation}\label{quadratic-ext-field}
Q(x) = (1+x_1)^2, \quad x\in\S^{d-1}, \quad d\geq 3.
\end{equation}
An external field given by a quadratic polynomial was first considered in \cite{bilo} for the case $d=3$ and Newton potential, see Proposition 5 and Theorems 7 and 8 there. Below we extend the corresponding statements from \cite{bilo} to the arbitrary dimension $d\geq3$, for the monic quadratic polynomial of the form appearing on the right hand side of $(\ref{quadratic-ext-field})$.

It is a straightforward calculation to verify that $Q(x)$ is a nonnegative convex increasing function on $[-1,1]$, also possessing rotational symmetry with respect to rotations about the polar axis. Therefore, by Proposition \ref{scssp}, the support of the equilibrium measure $\mu_Q$ for this external field will be a spherical cap $C_{S,\alpha}$. Following the established procedure, we first compute the Mhaskar-Saff $\cF$-functional, by substituting expression for the external field $(\ref{quadratic-ext-field})$ into $(\ref{msfunc})$.
\begin{theorem}\label{msf-quadratic-ext-field-theo}
In the case of the rational external field $(\ref{quadratic-ext-field})$, the Mhaskar-Saff $\cF$-functional for the spherical cap $C_{S,\alpha}$ is of the form
\begin{empheq}{align}\label{msf-quadratic-ext-field}
 \cF(C_{S,\alpha})  =  \frac {\sqrt{\pi}\,  \Gamma(d/2-1)} {2^{d-2}\,  \Gamma((d-1)/2)} \,  \bigg(\Beta\bigg(\cos^2\bigg(\frac {\alpha} {2} \bigg);  \frac{d}{2}-1, &  \frac{d}{2}  \bigg) \bigg)^{-1} \times \\ \nonumber 
 & \bigg\{ 1  + \frac {2^d\, \Gamma((d+3)/2)} {\sqrt{\pi} \, \Gamma(d/2+1)} \Beta \bigg( \cos^2\bigg(\frac {\alpha} {2} \bigg) ; \frac{d}{2}+1,\frac{d}{2}\bigg)\bigg\}.
\end{empheq}
\end{theorem}
\noindent The density of the corresponding equilibrium measure is found by applying Theorem \ref{theo3}.
\begin{theorem}\label{density-rational-ext-field}
If the external field $Q$ defined via $(\ref{extfieldexam})$, the support $S_Q$ is a spherical cap $C_{S,\alpha}$ centered at the South Pole, with $\alpha=\alpha_0\in(0,\pi)$. The 
angle $\alpha_0$ is defined as a unique solution of the equation
\begin{equation}\label{eqn-crit-pnts-quadratic-ext-field}
\frac {\sqrt{\pi}d(d-2)\, \Gamma(d/2-1)}  {2^d (d^2-1) \Gamma((d-1)/2)}  =  \cos^4\bigg(\frac {\alpha} {2} \bigg) \Beta \bigg( \cos^2\bigg(\frac {\alpha} {2} \bigg) ; \frac{d}{2}-1,\frac{d}{2}\bigg) - \Beta \bigg( \cos^2\bigg(\frac {\alpha} {2} \bigg) ; \frac{d}{2}+1, \frac{d}{2}\bigg).
\end{equation}
Let
\begin{empheq}{align}\label{F-quadratic-ext-field-theo}
F(\eta) =  - \frac  { 2\, \Gamma((d+3)/2)} { d(d-2)\pi^{(d+1)/2} }  \, \bigg\{& \left( \frac {1-\cos\alpha_0} {1-\cos\eta} \right)^{d/2} \, \sqrt{\frac{1-\cos\eta}{\cos\alpha_0-\cos\eta}}\, (1+\cos\alpha_0)^2 \\ \nonumber
& + 2(d-1)\, \Beta \bigg( \frac{\cos\alpha_0-\cos\eta}{1-\cos\eta} ; \frac{1}{2},\frac{d}{2}\bigg) \\ \nonumber
& - 2(d+1)\,(1-\cos\eta)\,  \Beta \bigg( \frac{\cos\alpha_0-\cos\eta}{1-\cos\eta} ; \frac{1}{2},\frac{d}{2}+1 \bigg) \\ \nonumber
& + \frac{d+3}{2}\,(1-\cos\eta)^2\,  \Beta \bigg( \frac{\cos\alpha_0-\cos\eta}{1-\cos\eta} ; \frac{1}{2},\frac{d}{2}+2 \bigg) \bigg\}, \quad \alpha_0 \leq \eta \leq \pi,
\end{empheq}
and
\begin{empheq}{align}\label{C-Q-quadratic-ext-field}
C_Q  =  \frac {\Gamma(d/2-1)} {2^{d-1}\,  \pi^{d/2}}   \bigg(\Beta\bigg(\cos^2\bigg(\frac {\alpha_0} {2} \bigg);  \frac{d-2}{2},  \frac{d}{2} \bigg) \bigg)^{-1} \bigg\{ 1  + \frac {2^d\, \Gamma((d+3)/2)} {\sqrt{\pi} \, \Gamma(d/2+1)} \Beta \bigg( \cos^2\bigg(\frac {\alpha_0} {2} \bigg) ; \frac{d}{2}+1,\frac{d}{2}\bigg)\bigg\}.
\end{empheq}
The density of the equilibrium measure $\mu_Q$ is given by
\begin{equation}\label{densExamp-rational-ext-field}
f(\eta)  =   C_Q\,  \bigg(\frac{1-\cos\alpha_0}{1-\cos\eta} \bigg)^{(d-1)/2} \, \sqrt{\frac{1-\cos\alpha_0}{\cos\alpha_0-\cos\eta}}  \, {}_2 F_1\bigg (1, \frac {d-1} {2};\frac {1} {2};  \frac{\cos\alpha_0-\cos\eta}{1-\cos\eta} \bigg) + F(\eta), \quad \alpha_0 \leq \eta \leq \pi.
\end{equation}
\end{theorem}

% --------------------PROOFS------------------------- PROOFS------------------------- PROOFS------------------------- PROOFS----------- PROOFS------------------

\section{Proofs}
\noindent{\bf Proof of Theorem \ref{theo1}.} The idea of the proof is to show that the equilibrium potential $U^{\mu_Q}$ is Lipschitz continuous on $C_{S,\alpha}$. 
If that is established, it will imply that the normal derivatives of $U^{\mu_Q}$ exist a.e. on  $C_{S,\alpha}$. Then the 
measure $\mu_Q$ can be recovered from its potential by the formula
\begin{equation}\label{meas-recov}
d\mu_Q = -\frac{1}{(d-2)\, \omega_d}\left(\frac{\partial U^{\mu_Q}}{\partial n_+}+\frac{\partial U^{\mu_Q}}{\partial n_-}\right)\,d\sigma := f(\theta_1)\, d\sigma,
\end{equation}
where $d\sigma$ is the Lebesgue surface measure on $\supp(\mu_Q)$, $n_{+}$  and $n_{-}$ are the inner and the outer normals to the cap $C_{S,\alpha}$. It is clear that 
the normal derivatives of $U^{\mu_Q}$ are bounded a.e. by the Lipschitz constant, and hence we 
obtain $f\in L_{\text{loc}}^\infty([\alpha,\pi])$. 

Our first step is to construct an extension 
of the external field $Q$ to $\S^{d-1}$ in such a way that the extremal measures for the cap $C_{S,\alpha}$ and 
the sphere $\S^{d-1}$ are the same. Recall that the external field $Q$ is a $C^2$ function on an open neighborhood $N$ of $S_Q$ in $\S^{d-1}$. We can adjust $Q$ in such a way that 
for the new external field $\widetilde{Q}$ one has
\[
U^{\mu_Q}(x)+\widetilde{Q}(x) = F_Q,  \quad x \in S_Q,
\]
\[
U^{\mu_Q}(x)+\widetilde{Q}(x) \geq F_Q,  \quad x \in \S^{d-1},
\]
and also $\widetilde{Q}\in C^2(\S^{d-1})$.
To show that it is indeed possible, we first remark that the external field $Q$ is rotationally symmetric with respect to the rotations about the polar axis. 
Therefore, $Q$ is a function of the polar angle $\theta_1$ only, that is $Q=Q(\theta_1)$. This symmetry 
is also inherited by potential, so that on the sphere $\S^{d-1}$ we have $U^{\mu_Q}(x)=U^{\mu_Q}(\theta_1)$, $x=(r,\theta_1,\ldots,\theta_{d-2},\varphi)\in\S^{d-1}$. Next, note that $N= \{ (r,\theta_1,\ldots,\theta_{d-2},\varphi): r=1,\ \ \gamma <  \theta_1 \leq \pi, \ \ 0\leq \theta_j \leq \pi, j=2,\ldots,{d-2}, \ \  0\leq \varphi \leq 2\pi  \}$, with some $\gamma\in(0,\alpha)$. Pick a number $\epsilon$ such that $\gamma < \epsilon < \alpha$. We define a new external field $\widetilde{Q}$ as follows: set $\widetilde{Q}(\theta_1)=Q(\theta_1)$, for $\epsilon < \theta_1 \leq \pi$, while on  $[0, \epsilon]$ we tweak $Q$ to $\widetilde{Q}$ in such a way that 
\[
U^{\mu_Q}(\theta_1)+\widetilde{Q}(\theta_1)\geq F_Q,
\]
and $\widetilde{Q}\in C^2(\S^{d-1})$. Applying Theorem 4.2.14 from \cite{bhs}, we infer that $\mu_{\widetilde{Q}}=\mu_Q$ and $F_{\widetilde Q}=F_Q$. 

Let $u$ and $v$ denote the equilibrium potentials for the minimum energy problem on $C_{S,\alpha}$ and $\S^{d-1}$, 
respectively. Since the equilibrium measure is the same for those two sets, it immediately follows that $u=v$. 
Now observe that the spherical cap $C_{S,\alpha}$ is a part of the sphere $\S^{d-1}$, which is a closed smooth surface. Thus 
we can invoke the result of G\"otz \cite{gotz} to conclude that $v$ 
is Lipschitz continuous in an open neighborhood $\mathcal{U}$ of $\S^{d-1}$. Hence $U^{\mu_Q}$ is Lipschitz continuous on $C_{S,\alpha}$.

We finish the proof by justifying formula $(\ref{meas-recov})$. 
\begin{lemma}\label{meas-recov-lemma}
Let $U^\mu$ be the potential of a measure $\mu$ in a domain $G\subset\R^d$. Suppose that the intersection of $\supp(\mu)$ and the domain $G$ is a connected $C^1$ hypersurface $\Omega$. Suppose also that the potential $U^\mu$ is Lipschitz continuous on an open neighborhood of $\Omega$. Then on $\Omega$ the measure $\mu$ is 
locally absolutely continuous with respect to the Lebesgue surface measure $d\sigma$, and we have the representation
\begin{equation}\label{meas-recov-form}
d\mu = -\frac{1}{(d-2)\, \omega_d}\left(\frac{\partial U^{\mu}}{\partial n_+}+\frac{\partial U^{\mu}}{\partial n_-}\right) d\sigma,
\end{equation}
where $n_{+}$  and $n_{-}$ are the inner and the outer normals to $\Omega$.
\end{lemma}

\begin{proof} For the case $d=3$ formula $(\ref{meas-recov-form})$ is proved in \cite[p. 164]{kel}. Also, when the measure $\mu$ is supported on a hyperplane,
expression $(\ref{meas-recov-form})$ was derived in  \cite[Lemma 3.1, p. 48]{wise}.

We begin by observing that there is a neighborhood of $\Omega$ where $\Omega$ separates $G$ into two 
pieces. We will be denoting the intersection of $G$ with that neighborhood again by $G$. 
We next pick an interior point $x\in \Omega$ and consider a small ball 
$B(r,x)$ centered at $x$, where $r>0$ is chosen such that $\overline{B(r,x)}\subset G$.
We then choose a positive side of $\Omega$ and denote the normal in that direction by $n_{+}$,
while the normal to a negative side of $\Omega$ will be denoted by $n_{-}$. We will also use the subscripts
$+$ and $-$ to distinguish the subsets of $G$ and $B(r,x)$ that lie on positive and negative sides of $\Omega$.

Let $u=U^{\mu}$ and $v(y)=1/|x-y|^{d-2}$. Observe that when $x$ is fixed, the function $v(y)$ is harmonic
for $y\neq x$. In particular, it is harmonic on a neighborhood of $G_{+}\setminus B_{+}(r,x)$. We also know that 
$u$ is harmonic in $G\setminus\Omega$. Therefore, there exists a compact set with a neighborhood
where $u$ and $v$ are both harmonic. Let $V_\varepsilon:=\{x\in G: \dist(x,\Omega) < \varepsilon \}$ be a small open neighborhood of $\Omega$, and set 
$K_\varepsilon:=(G_{+}\setminus B_{+}(r,x))\setminus V_\varepsilon$. The Green's identity \cite[p. 22]{hayken} states
that
\begin{equation}\label{green1}
\int_{K_\varepsilon} (u \Delta v - v \Delta u)\, dy = \int_{\partial K_\varepsilon} \left( u \frac{\partial v}{\partial n} - v \frac{\partial u}{\partial n} \right) d\sigma,
\end{equation}
where $\partial/\partial n$ denotes the inward normal derivative on $G\setminus B(r,x)$. As both $u$ and $v$ are harmonic in a neighborhood of $K_\varepsilon$, the 
left hand side of $(\ref{green1})$ is zero. Therefore,
\begin{equation}\label{green2}
\int_{\partial K_\varepsilon} \left( u \frac{\partial v}{\partial n} - v \frac{\partial u}{\partial n} \right) d\sigma = 0.
\end{equation}
Since the potential $u$ is Lipschitz continuous, its normal derivative is bounded a.e. by a Lipschitz constant. Therefore, 
passing to the limit $\varepsilon\rightarrow 0+$ in $(\ref{green2})$, and applying the Dominated Convergence Theorem, we deduce that
\begin{equation}\label{green3}
\int_{\partial (G_{+}\setminus B_{+}(r,x))} \left( u \frac{\partial v}{\partial n} - v \frac{\partial u}{\partial n} \right) d\sigma = 0.
\end{equation}
We proceed by splitting the domain of integration in $(\ref{green3})$ into a component that lies on $\Omega$, and the two 
components that do not. On the positive side of $\Omega$ relation $(\ref{green3})$ reads
\begin{equation}\label{green4}
\int_{\partial (G_{+}\setminus B_{+}(r,x))\setminus\Omega} \left( u \frac{\partial v}{\partial n_{+}} - v \frac{\partial u}{\partial n_{+}} \right) d\sigma + \int_{\partial (G_{+}\setminus B_{+}(r,x))\cap \Omega} \left( u \frac{\partial v}{\partial n_{+}} - v \frac{\partial u}{\partial n_{+}} \right) d\sigma= 0.
\end{equation}
By similar considerations, working with the negative side of $\Omega$, we obtain
\begin{equation}\label{green5}
\int_{\partial (G_{-}\setminus B_{-}(r,x))\setminus\Omega} \left( u \frac{\partial v}{\partial n_{-}} - v \frac{\partial u}{\partial n_{-}} \right) d\sigma + \int_{\partial (G_{-}\setminus B_{-}(r,x))\cap \Omega} \left( u \frac{\partial v}{\partial n_{-}} - v \frac{\partial u}{\partial n_{-}} \right) d\sigma= 0.
\end{equation}
Adding the right hand sides of $(\ref{green4})$ and $(\ref{green5})$, and observing that the normal 
derivatives of $v$ on $\Omega\setminus B(r,x)$ are of opposite values, we infer
\begin{equation}\label{green6}
\int_{(G\setminus B(r,x))\cap\Omega} v\, \left( \frac{\partial u}{\partial n_{+}} + \frac{\partial u}{\partial n_{-}} \right) d\sigma = \int_{\partial G} \left( u \frac{\partial v}{\partial n} - v \frac{\partial u}{\partial n} \right) d\sigma + \int_{\partial B(r,x)} \left( u \frac{\partial v}{\partial n} - v \frac{\partial u}{\partial n} \right) d\sigma.
\end{equation}
We now deal with the first integral on the right hand side of $(\ref{green6})$. Observe that in a neighborhood of $\partial G$, the potential $u$ does not depend on the choice of $x$, and the function $v$ is clearly harmonic as a function of $x$. It then follows that the first integral represents a function of $x$, which is harmonic in a neighborhood of $\partial G$, and which will be denoted by $g(x)$.

We now turn to the second integral on the right hand side of $(\ref{green6})$. First note that
\begin{equation}\label{der-v}
\frac{\partial v}{\partial n} = (d-2)\, \frac{1}{r^{d-1}}, \quad y\in\partial B(r,x).
\end{equation}
Using $(\ref{der-v})$ and the Mean Value Theorem for harmonic functions \cite{land}, we obtain
\begin{empheq}{align}\label{green7}
\int_{\partial B(r,x)} u \frac{\partial v}{\partial n} d\sigma & =  \frac {d-2}{r^{d-1}}\, \int_{\partial B(r,x)} u(y)\, d\sigma \\ \nonumber
	& = (d-2)\, \omega_d\, u(x). \nonumber
\end{empheq}
Now recall that the potential $u$ is assumed to be Lipschitz continuous, with the Lipschitz constant which we will denote by $L$. Then it follows that the normal derivative 
of $u$ will be bounded a.e. by $L$. With that in hand, we have the following estimate
\begin{empheq}{align}\label{green8}
\left| \int_{\partial B(r,x)} v \frac{\partial u}{\partial n} \, d\sigma \right| & = \left| \int_{\partial B(r,x)} \frac{1}{r^{d-2}}\, \frac{\partial u}{\partial n} \, d\sigma \right| \\ \nonumber
		& \leq \frac{1}{r^{d-2}}\, \int_{\partial B(r,x)} \left| \frac{\partial u}{\partial n} \right| \, d\sigma \\ \nonumber
		& \leq \frac{L}{r^{d-2}}\, \int_{\partial B(r,x)}  d\sigma \\ \nonumber
		& = \omega_d\, L\, r.
\end{empheq}
Estimate $(\ref{green8})$ shows that
\begin{equation}\label{green9}
\lim_{r\rightarrow0+} \int_{\partial B(r,x)} v \frac{\partial u}{\partial n} \, d\sigma = 0.
\end{equation}
Passing to the limit $r\to 0+$ in left hand side of $(\ref{green6})$, and noting that $\chi_{B(r,x)}\to 0$ a.e. as $r\to 0+$, by the Dominated 
Convergence Theorem we obtain
\begin{empheq}{align}\label{limlhs}
\lim_{r\rightarrow0+} \int_{(G\setminus B(r,x))\cap\Omega} v\, \left( \frac{\partial u}{\partial n_{+}} + \frac{\partial u}{\partial n_{-}} \right) d\sigma & = \lim_{r\rightarrow0+} \int   \chi_{G\cap\Omega} \, (1-\chi_{B(r,x)})\, v\left( \frac{\partial u}{\partial n_{+}} + \frac{\partial u}{\partial n_{-}} \right) d\sigma \\ \nonumber
& = \int  \chi_{G\cap\Omega} \, v\left( \frac{\partial u}{\partial n_{+}} + \frac{\partial u}{\partial n_{-}} \right) d\sigma \\\nonumber
& = \int_{G\cap\Omega} v\left( \frac{\partial u}{\partial n_{+}} + \frac{\partial u}{\partial n_{-}} \right) d\sigma.
\end{empheq}
Collecting $(\ref{limlhs})$, $(\ref{green7})$ and $(\ref{green9})$, we obtain
\begin{equation*}
\int_{G\cap\Omega} \left( \frac{\partial u}{\partial n_{+}} + \frac{\partial u}{\partial n_{-}} \right)\, \frac{d\sigma}{|x-y|^{d-2}} = (d-2)\, \omega_d\, u(x) + g(x).
\end{equation*}
The uniqueness part of the Riesz Decomposition Theorem \cite[Theorem $1.22'$, p. 104]{land} then yields that on $\Omega$ the measure $\mu$ is given by the expression
\begin{equation*}
d\mu = - \frac{1}{(d-2)\, \omega_d}\left(\frac{\partial U^{\mu}}{\partial n_+}+\frac{\partial U^{\mu}}{\partial n_-}\right) d\sigma,
\end{equation*}
as desired.
\end{proof}
\noindent The proof of the theorem is now complete.

\qed

% -------------------- Theorem 1.3 ---------------------------------- Theorem 1.3 ------------------------- Theorem 1.3 ---------------------------------- Theorem 1.3 ------------------------

\noindent{\bf Proof of Theorem \ref{theo2}.} Let the support of the extremal measure $\mu_Q$ be a spherical cap centered at the North Pole, that is $S_Q = C_{N,\alpha}$. From Theorem \ref{theo1} we know 
that $d\mu_Q = f(\theta_1)\, d\sigma_d$, where $f\in L_{\text{loc}}^\infty([0,\alpha])$. 

In what follows, we will need an expression for the distance between a point on a sphere $\S^{d-1}$ and another point in the space $\R^d$, which is not on the surface of the sphere $\S^{d-1}$. Let
\begin{align*}
x & =(x_1,x_2,x_3,\ldots,x_d) \\
	& = (r \cos \theta_1,r\sin \theta_1\cos \theta_2,r\sin\theta_1 \sin\theta_2\cos\theta_3, \ldots,r\sin\theta_1\sin\theta_2\ldots\sin\theta_{d-2} \sin\varphi)\in\R^d,\\
y & = (y_1,y_2,y_3,\ldots,x_d) \\
	& = ( \cos \eta_1,\sin \eta_1\cos \eta_2,\sin\eta_1 \sin\eta_2 \cos\eta_3, \ldots,\sin\eta_1\sin\eta_2\ldots\sin\eta_{d-2} \sin\psi)\in\S^{d-1},
\end{align*}
be such two points, written in hyperspherical coordinates. Then, for the inner product of $x$ and $y$, we obtain
\begin{empheq}{align*}\label{ip}
\langle x,y \rangle = \sum_{j=1}^d x_j y_j & = x_1y_1 + \sum_{j=2}^d x_j y_j \\ \nonumber
			& = r \cos\theta_1\cos\eta_1 + r \sin\theta_1\sin\eta_1 \langle \overbar{x},\overbar{y}\rangle,
\end{empheq}
where 
\begin{align*}
\overbar{x} & = (\cos \theta_2, \sin\theta_2\cos\theta_3, \ldots, \sin\theta_2  \ldots \sin\theta_{d-3} \cos\theta_{d-2} \cos\varphi,  \sin\theta_2  \ldots \sin\theta_{d-3} \cos\theta_{d-2} \sin\varphi) \in\S^{d-2},\\
\overbar{y} & = (\cos \eta_2, \sin\eta_2\cos\eta_3, \ldots, \sin\eta_2  \ldots \sin\eta_{d-3} \cos\eta_{d-2} \cos\psi,  \sin\eta_2  \ldots \sin\eta_{d-3} \cos\eta_{d-2} \sin\psi) \in\S^{d-2}.
\end{align*}
Therefore, for the distance $|x-y|$, we obtain
\begin{align*}
 |x-y|^2 & = |x|^2 + |y|^2 - 2 \langle x,y \rangle \\
 		& = r^2 + 1 - 2r(\cos\theta_1 \cos\eta_1 + \sin\theta_1 \sin\eta_1 \langle \overbar{x},\overbar{y}\rangle ) \\
		& = r^2 + 1 - 2r\lambda,
\end{align*}
where $\lambda=\cos\theta_1 \cos\eta_1 + \sin\theta_1 \sin\eta_1\, \langle\overbar{x},\overbar{y}\rangle$.
Thus, the potential $U^{\mu_Q}$ becomes
\begin{empheq}{align*}
U^{\mu_Q}(x) & = \int_{C_{N,\alpha}} \frac {d\mu_Q(y)} {|x-y|^{d-2}} \\ 
			 & = \int_0^\alpha  f(\eta_1) \, \sin^{d-2} \eta_1 \, d\eta_1 \int_{\S^{d-2}} \frac { d \sigma_{d-1}(\overbar{y}) } {(r^2 + 1 - 2r\lambda)^{(d-2)/2} }.
\end{empheq}
On the surface of the sphere $\S^{d-1}$ we have $r=1$, so that
\begin{equation}\label{pinspc}
U^{\mu_Q}(x) = \int_0^\alpha  f(\eta_1) \, \sin^{d-2}\eta_1 \, d\eta_1 \int_{\S^{d-2}} \frac { d \sigma_{d-1}(\overbar{y}) } {(2 - 2\lambda)^{(d-2)/2} }, \quad x\in C_{N,\alpha}.
\end{equation}
We will need the following proposition, which is a special case of the Funk-Hecke theorem \cite[p. 247]{be2}.
\begin{proposition}\label{funkhecke}
If $f$ is integrable on $[-1,1]$ with respect to the weight $(1-t^2)^{(d-3)/2}$, and $y$ is an arbitrary fixed point on the sphere $S^{d-1}$, then
\begin{equation}\label{fhformula}
\int_{\S^{d-1}} f (\langle x,y \rangle ) \, d\sigma_d(x) = \frac {2 \pi^{(d-1)/2}} {\Gamma ((d-1)/2)} \, \int_{-1}^1 f(t) \, (1-t^2)^{(d-3)/2} \, dt.
\end{equation}
\end{proposition}

\noindent Applying Proposition \ref{funkhecke} to the inner integral in $(\ref{pinspc})$, we thus obtain
\begin{align*}
 \int_{\S^{d-2}} \frac { d \sigma_{d-1}(\overbar{y}) } {(2 - 2\lambda)^{(d-2)/2} } & = \int_{\S^{d-2}} \frac { d \sigma_{d-1}(\overbar{y}) } {(2 - 2(\cos\theta_1 \cos\eta_1 + \sin\theta_1 \sin\eta_1 (\overbar{x},\overbar{y})))^{(d-2)/2} } \\
 & =  \frac {2 \pi^{(d-2)/2}} {\Gamma ((d-2)/2)} \, \int_0^\pi \frac {\sin^{d-3}\xi \,d\xi} {(2 - 2(\cos\theta_1 \cos\eta_1 + \sin\theta_1 \sin\eta_1 \cos\xi))^{(d-2)/2}}.
\end{align*}
Hence, for the potential $U^{\mu_Q}$ on the spherical cap ${C_{N,\alpha}}$, we finally obtain
\begin{align*}
U^{\mu_Q}(\theta_1) & =   \frac {2 \pi^{(d-2)/2}} {\Gamma ((d-2)/2)} \,  \int_0^\alpha  f(\eta_1) \, \sin^{d-2}\eta_1 \, d\eta_1 \, \int_0^\pi \frac {\sin^{d-3} \xi \,d\xi} {(2 - 2(\cos\theta_1 \cos\eta_1 + \sin\theta_1 \sin\eta_1 \cos\xi))^{(d-2)/2}} \\
& =   \frac {2 \pi^{(d-2)/2}} {\Gamma ((d-2)/2)}  \,  \int_0^\alpha  f(\eta_1) \, \sin^{d-2}\eta_1 \, d\eta_1 \, \int_0^\pi \frac { \sin^{d-3} \xi\,d\xi} {(2 - 2\gamma)^{(d-2)/2}}, \quad 0 \leq \theta_1 \leq \alpha,
\end{align*}
where $\gamma$ is defined as
\begin{equation}\label{gam}
\gamma = \cos\theta_1 \cos\eta_1 + \sin\theta_1 \sin\eta_1 \cos\xi.
\end{equation}
Therefore, integral equation $(\ref{ieec})$ now assumes the form
\begin{equation}\label{ie1}
  \frac {2 \pi^{(d-2)/2}} {\Gamma ((d-2)/2)}  \,  \int_0^\alpha  f(\eta_1)\, \sin^{d-2} \eta_1 \, d\eta_1 \, \int_0^\pi \frac { \sin  ^{d-3} \xi\,d\xi} {(2 - 2\gamma)^{(d-2)/2}} = F_Q - Q(\theta_1), \quad 0 \leq \theta_1 \leq \alpha.
\end{equation}
Integral equation $(\ref{ie1})$, in the case $d=3$, was first obtained and solved by Collins \cite{col}. A generalization of the results of \cite{col} to the case of $d$ dimensions, $d\geq 3$, was considered by Shail \cite{shail}. However, some of the arguments employed in \cite{shail}, in particular, those used in derivation of a special case of integral equation $(\ref{ie1})$, are not mathematically rigorous. 

Letting
\[
a:= 2\sin\left( \frac {\theta_1} {2} \right)  \cos\bigg( \frac {\eta_1} {2} \bigg),
\]
\[
b:= 2 \sin\bigg( \frac {\eta_1} {2} \bigg) \cos\left( \frac {\theta_1} {2} \right),
\] 
using elementary trigonometric identities, we can easily see that
\begin{align*}
2 - 2\gamma & = 2 - 2 (\cos\theta_1 \cos\eta_1 + \sin\theta_1 \sin\eta_1 \cos\xi) \\
			& = a^2 + b^2 - 2ab\cos\xi.
\end{align*}
Observe that as $0<\alpha\leq\pi$, it is clear that $a>0$ and $b>0$ for all and $0\leq \theta_1 \leq \alpha$ and $0\leq \eta_1 \leq\alpha$. 
The next step is to further transform the kernel of the integral equation $(\ref{ie1})$. Namely, it is facilitated via the following
\begin{lemma}\label{copgen}
If $a$ and $b$ are positive numbers, $a\neq b$, and $q \geq 0$, then
\begin{equation}\label{copgenf}
\int_0^{\pi} \frac {\sin^{2q}\xi \,d\xi} {(a^2+b^2-2ab\cos\xi)^{q+\frac{1}{2}}} = \frac{2}{a^{2q}\, b^{2q}} \int_0^{\min{(a,b)}} \frac {t^{2q}\,dt} {\sqrt{a^2-t^2} \sqrt{b^2-t^2}}.
\end{equation}
\end{lemma}
\noindent We remark that Lemma \ref{copgen} is generalization of a result, obtained by Copson \cite{cop} when $q=0$. Lemma \ref{copgen} is implicitly mentioned in \cite{shail}, although with an incorrect numerical coefficient. In \cite{shail}, the author derives $(\ref{copgenf})$ using identities for the Bessel functions. However, the form of the result suggests that its proof is independent of any special function. We present such a proof below.
\begin{proof}
The proof hinges on the following identity, obtained by Kahane \cite{kah}.
\begin{proposition}\label{kahlemma}
Let $a$ and $b$ be positive numbers such that $a\neq b$, $\tau\in(0,1)$, $\upsilon\in\mathbb C$ with $\text{Re}(\upsilon)\geq0$, and $u$ a real number with $|u|\leq 1$. Then
\begin{empheq}{align}\label{kahformula}
 \frac{(ab)^\upsilon} {(a^2+b^2-2abu)^{\tau+\upsilon}} & = \frac {\Gamma(\tau) \Gamma(\upsilon+1)} {\Gamma(\tau+\upsilon)} \frac {2\sin(\tau \pi)} {\pi} \times   \\ \nonumber
& \int_0^{\min{(a,b)}} \frac {1-(t^2/ab)^2} {(1+(t^2/ab)^2-2(t^2/ab)u)^{\upsilon+1}} \left( \frac {t^2} {ab} \right)^\upsilon \frac {t^{2\tau-1} \,dt} {(a^2-t^2)^{\tau} (b^2-t^2)^{\tau}}.
\end{empheq}
\end{proposition}
\noindent Setting $\tau=1/2$, $\upsilon=q\geq0$ in Proposition \ref{kahlemma}, and using Fubini's theorem, we rewrite the left hand side of $(\ref{copgenf})$ as
\begin{empheq}{align}\label{kern1}
\int_0^{\pi} \frac {\sin^{2q}\xi \,d\xi} {(a^2+b^2-2ab\cos\xi)^{q+\frac{1}{2}}}  =  \frac{2}{a^{2q}\, b^{2q}} \frac {\Gamma(q+1)} {\sqrt{\pi} \Gamma(q+1/2)} & \int_0^{\min{(a,b)}}  \frac {(1-(t^2/ab)^2)\, t^{2q}\,dt} { {\sqrt{a^2-t^2} \sqrt{b^2-t^2}}} \times \\ \nonumber
& \int_0^\pi \frac{\sin^{2q}\xi \, d\xi} {(1+(t^2/ab)^2-2(t^2/ab)\cos\xi)^{q+1}}.
\end{empheq}
Next, we show that
\begin{equation}\label{ltint}
\int_0^\pi \frac{\sin^{2q}\xi \, d\xi} {(1+(t^2/ab)^2-2(t^2/ab)\cos\xi)^{q+1}} = \frac {1} {1-(t^2/ab)^2} \frac {\sqrt{\pi}\, \Gamma(q+1/2)} {\Gamma(q+1)}.
\end{equation}
The integral of a type appearing on the left hand side of $(\ref{ltint})$ was previously considered in \cite[p. 400]{land}. It was shown that
\begin{equation}\label{lankint}
\int_0^\pi \frac {\sin^{p-2}\xi \, d\xi} {(1+\rho^2-2\rho\cos\xi)^{p/2}} = \frac {1} {\rho^{p-2}(\rho^2-1)} \int_0^\pi \sin^{p-2}\xi \,d\xi,
\end{equation}
where $\rho\geq 1$, and $p\geq 3$ was assumed to be an integer. A careful analysis of the evaluation of integral $(\ref{lankint})$ in \cite[p. 400]{land} shows that, in fact, $(\ref{lankint})$ holds true
for any $p\geq2$.
We hence transform the left hand side of $(\ref{ltint})$ as follows,
\begin{align*}
\int_0^\pi \frac{\sin^{2q}\xi \, d\xi} {(1+(t^2/ab)^2-2(t^2/ab)\cos\xi)^{q+1}} & = \frac {1} {1-(t^2/ab)^2} \int_0^\pi \sin^{2q}\xi \, d\xi \\
			& = \frac {1} {1-(t^2/ab)^2} 2^{2q} \, \Beta(q+1/2,q+1/2) \\
			& = \frac {1} {1-(t^2/ab)^2}  \frac {\sqrt{\pi}\, \Gamma(q+1/2)} {\Gamma(q+1)},
\end{align*}
as claimed.
Therefore, upon inserting $(\ref{ltint})$ into $(\ref{kern1})$, we obtain desired representation $(\ref{copgenf})$.
\end{proof}
\noindent Setting $q=(d-3)/2$ in Lemma \ref{copgen}, integral equation $(\ref{ie1})$ is transformed into
\begin{equation*}
\frac {4\, \pi^{(d-2)/2}} { \Gamma((d-2)/2)} \,  \int_0^\alpha  f(\eta_1) \sin^{d-2}\eta_1  \, d\eta_1 \, \frac{1}{a^{d-3}\, b^{d-3}} \int_0^{\min{(a,b)}} \frac {t^{d-3}\,dt} {\sqrt{a^2-t^2} \sqrt{b^2-t^2}} = F_Q - Q(\theta_1), \quad 0 \leq \theta_1 \leq \alpha. 
\end{equation*}
To simplify notation, we will use $\eta$ and $\theta$ instead of $\eta_1$ and $\theta_1$, respectively. Then, the last equation reads
\begin{equation}\label{ie2}
\frac{4\, \pi^{(d-2)/2}} { \Gamma((d-2)/2)}  \,  \int_0^\alpha  f(\eta)\, \sin^{d-2} \eta \, d\eta \, \frac{1}{a^{d-3}\, b^{d-3}} \int_0^{\min{(a,b)}} \frac {t^{d-3}\,dt} {\sqrt{a^2-t^2} \sqrt{b^2-t^2}} = F_Q - Q(\theta), \quad 0 \leq \theta \leq \alpha, 
\end{equation}
where 
\[
a= 2\sin\left( \frac {\theta} {2} \right)  \cos\bigg( \frac {\eta} {2} \bigg),
\]
\[
b= 2 \sin\bigg( \frac {\eta} {2} \bigg) \cos\left( \frac {\theta} {2} \right).
\] 
One can easily check that $a<b$ for $\theta<\eta$, while for $\theta>\eta$, we have $a>b$. Splitting the interval of integration of the outer integral in the 
left hand side of equation $(\ref{ie2})$, we rewrite equation $(\ref{ie2})$ as follows,
\begin{align*}
\int_0^\theta f(\eta) \sin^{d-2} \eta\, d\eta \, \frac{1}{a^{d-3}\, b^{d-3}}  \int_0^b \frac {t^{d-3}\,dt} {\sqrt{a^2-t^2} \sqrt{b^2-t^2}} + & \int_\theta^\alpha f(\eta)  \sin^{d-2} \eta\, d\eta \, \frac{1}{a^{d-3}\, b^{d-3}}  \int_0^a \frac {t^{d-3}\,dt} {\sqrt{a^2-t^2} \sqrt{b^2-t^2}} \\ = & \frac  { \Gamma((d-2)/2)} {4\, \pi^{(d-2)/2}}  \, (F_Q - Q(\theta)), \quad 0 \leq \theta \leq \alpha.
\end{align*}
Introducing the substitution
\[
t = 2 \cos\bigg(\frac {\theta} {2}\bigg) \cos\bigg( \frac {\eta} {2}\bigg) \tan\bigg( \frac {\zeta} {2}\bigg),
\]
we can further transform the last integral equation into
\begin{empheq}{align}\label{ie3}
  \int_0^\theta f(\eta) \sin \eta\, \cos^{d-3}(\eta/2) \, d\eta  & \int_0^\eta   \frac {\tan^{d-3}(\zeta/2) \,d\zeta} {\sqrt{\cos \zeta - \cos\theta} \sqrt{\cos \zeta - \cos \eta}}  \\ \nonumber
  + \int_\theta^\alpha  f(\eta) \sin \eta\, \cos^{d-3}(\eta/2) \, d\eta & \int_0^\theta \frac {\tan^{d-3}(\zeta/2) \,d\zeta} {\sqrt{\cos \zeta - \cos\theta} \sqrt{\cos \zeta - \cos \eta}} \\ \nonumber 
& = \frac  { \Gamma((d-2)/2)} {2\, \pi^{(d-2)/2}}  \,  \sin^{d-3}\bigg(\frac{\theta}{2}\bigg) \, (F_Q - Q(\theta)), \quad 0 \leq \theta \leq \alpha.
\end{empheq}
Inverting the order of integration in the first integral in the left hand side of $(\ref{ie3})$, we recast equation $(\ref{ie3})$ into
\begin{equation}\label{ie4}
 \int_0^\theta  \frac {\tan^{d-3}(\zeta/2) \,d\zeta} {\sqrt{\cos \zeta - \cos\theta} } \int_\zeta^\alpha  \frac {f(\eta) \sin \eta\,  \cos^{d-3}(\eta/2) \, d\eta} {\sqrt{\cos \zeta - \cos \eta}}  = \frac { \Gamma((d-2)/2)} {2\, \pi^{(d-2)/2}}  \,  \sin^{d-3}\bigg(\frac{\theta}{2}\bigg) \, (F_Q - Q(\theta)), \quad 0 \leq \theta \leq \alpha.
\end{equation}
Let
\begin{equation}\label{ie5}
S(\zeta) =  \int_\zeta^\alpha  \frac {f(\eta) \sin \eta\, \cos^{d-3}(\eta/2) \, d\eta} {\sqrt{\cos \zeta - \cos \eta }}, \quad 0 \leq \zeta \leq \alpha.
\end{equation}
Equation $(\ref{ie4})$ then becomes
\begin{equation}\label{ie6}
 \int_0^\theta \frac {S(\zeta)\,\tan^{d-3}(\zeta/2) \,d\zeta} {\sqrt{\cos \zeta - \cos\theta} } = \frac  { \Gamma((d-2)/2)} {2\, \pi^{(d-2)/2}}  \,  \sin^{d-3}\bigg(\frac{\theta}{2}\bigg) \, (F_Q - Q(\theta)), \quad 0 \leq \theta \leq \alpha.
\end{equation}
Equation $(\ref{ie6})$ is an Abel type integral equation with respect to $S(\zeta)\,\tan^{d-3}(\zeta/2)$. Since $Q\in C^2$, the solution of $(\ref{ie6})$ is \cite[p. 50, \# 23]{polman}
\begin{equation*}
S(\zeta) =\frac  { \Gamma((d-2)/2)} {2\, \pi^{d/2}}  \, \cot^{d-3}\bigg(\frac{\zeta}{2}\bigg) \frac {d} {d\zeta} \int_0^\zeta \frac {(F_Q - Q(\theta))\,  \sin^{d-3}(\theta/2) \sin\theta \, d\theta} {\sqrt{\cos\theta - \cos \zeta}}, \quad 0 \leq \zeta \leq \alpha.
\end{equation*}
Observing that equation $(\ref{ie5})$ is again an Abel type integral equation with respect to $f(\eta) \sin \eta\, \cos^{d-3}(\eta/2)$, we solve it and obtain
\begin{equation}\label{ie7}
f(\eta) = - \frac {1} {\pi} \frac {1} {\sin \eta}\, \sec^{d-3}\bigg(\frac{\eta}{2}\bigg) \, \frac {d} {d\eta} \int_\eta^\alpha \frac {S(\zeta) \sin \zeta\, d\zeta} {\sqrt{\cos \eta - \cos \zeta}}, \quad 0  \leq \eta \leq \alpha.
\end{equation}
Denote
\begin{equation}\label{ie8}
g(\zeta) = \cot^{d-3}\bigg(\frac{\zeta}{2}\bigg)\, \frac {d} {d\zeta} \int_0^\zeta \frac {Q(\theta)\,  \sin^{d-3}(\theta/2) \sin\theta \, d\theta} {\sqrt{\cos\theta - \cos \zeta}}, \quad 0 \leq \zeta \leq \alpha.
\end{equation}
and let
\begin{equation}\label{ie9}
F(\eta) = \frac  { \Gamma((d-2)/2)} {2\, \pi^{(d+2)/2}}  \, \frac {1} {\sin \eta} \,\sec^{d-3}\bigg(\frac{\eta}{2}\bigg) \, \frac {d} {d\eta} \int_\eta^\alpha \frac {g(\zeta) \sin \zeta \, d\zeta} {\sqrt{\cos \eta - \cos \zeta}}, \quad 0  \leq \eta \leq \alpha.
\end{equation}
In view of $(\ref{ie8})$ and $(\ref{ie9})$, expression for the density $(\ref{ie7})$ takes the form
\begin{empheq}{align}\label{ie10}
f(\eta)  = &  - \frac  { F_Q\, \Gamma((d-2)/2)} {2\, \pi^{(d+2)/2}}  \, \frac {\sec^{d-3}(\eta/2)} {\sin \eta}  \frac {d} {d\eta} \int_\eta^\alpha   \frac {\sin\zeta\, \cot^{d-3}(\zeta/2)} {\sqrt{\cos \eta - \cos \zeta}} \bigg\{ \frac {d} {d\zeta} \int_0^\zeta \frac{\sin^{d-3} (\theta/2) \, \sin\theta\,d\theta} {\sqrt{\cos \theta - \cos \zeta}} \bigg\}\, d\zeta  \\ \nonumber 
& + F(\eta), \quad  0  \leq \eta \leq \alpha.
\end{empheq}

\noindent It is not hard to see that
\begin{equation*}
 \int_0^\zeta \frac{\sin^{d-3} (\theta/2) \, \sin\theta\,d\theta} {\sqrt{\cos \theta - \cos \zeta}} = \sqrt{2} \, \Beta\left( \frac {d-1} {2}, \frac {1} {2} \right) \, \sin^{d-2} \left( \frac {\zeta} {2} \right).
\end{equation*}
Upon differentiating last expression with respect to $\zeta$, we find that
\begin{equation*}
 \frac {d} {d\zeta} \int_0^\zeta \frac{\sin^{d-3} (\theta/2) \, \sin\theta\,d\theta} {\sqrt{\cos \theta - \cos \zeta}} = \frac {d-2} {\sqrt{2}} \, \Beta\left( \frac {d-1} {2}, \frac {1} {2} \right) \, \sin^{d-3} \left( \frac {\zeta} {2} \right) \,  \cos \left( \frac {\zeta} {2} \right).
\end{equation*}
Using the elementary transformations, one can show that
\begin{equation*}
 \int_\eta^\alpha   \frac {\sin\zeta\, \cot^{d-3}(\zeta/2) \, \, \sin^{d-3} (\zeta/2) \,  \cos(\zeta/2)} {\sqrt{\cos \eta - \cos \zeta}} \, d\zeta = \sqrt{2} \,  \cos^{d-1} \left( \frac {\eta} {2} \right) \, \Beta\bigg(\frac {\cos\eta - \cos\alpha} {1 + \cos\eta}; \frac{1}{2}, \frac{d}{2} \bigg).
\end{equation*}
Hence, we can conclude that
\begin{align*}
\int_\eta^\alpha   \frac {\sin\zeta\, \cot^{d-3}(\zeta/2)} {\sqrt{\cos \eta - \cos \zeta}} \bigg\{ \frac {d} {d\zeta} \int_0^\zeta  &  \frac{\sin^{d-3} (\theta/2) \, \sin\theta\,d\theta} {\sqrt{\cos \theta - \cos \zeta}}  \bigg\}\, d\zeta \\ 
& = (d-2) \,   \Beta\left( \frac {d-1} {2}, \frac {1} {2} \right) \,  \cos^{d-1} \left( \frac {\eta} {2} \right) \,   \Beta\bigg(\frac {\cos\eta - \cos\alpha} {1 + \cos\eta}; \frac{1}{2}, \frac{d}{2} \bigg).
\end{align*}
Differentiating the latter, and  simplifying, we see that 
\begin{empheq}{align}\label{integ2}
 \frac {d} {d\eta} \int_\eta^\alpha   \frac {\sin\zeta\, \cot^{d-3}(\zeta/2)} {\sqrt{\cos \eta - \cos \zeta}} \bigg\{ \frac {d} {d\zeta} \int_0^\zeta   \frac{\sin^{d-3} (\theta/2) \, \sin\theta\,d\theta} {\sqrt{\cos \theta - \cos \zeta}} \bigg\}\, d\zeta & \\ \nonumber 
  = - \frac {d-2} {2} \,  \Beta\left( \frac {d-1} {2}, \frac {1} {2} \right) \, \sin\eta \, & \bigg\{  \frac {d-1} {2} \,  \cos^{d-3} \left( \frac {\eta} {2} \right) \,   \Beta\bigg(\frac {\cos\eta - \cos\alpha} {1 + \cos\eta}; \frac{1}{2}, \frac{d}{2} \bigg) \\ \nonumber
 & + \frac {\cos^{d-1}(\alpha/2)} {\cos^2(\eta/2) } \,  \sqrt{\frac{1+\cos\alpha}{\cos\eta-\cos\alpha}} \bigg\} .
\end{empheq}

\noindent Inserting $(\ref{integ2})$ into $(\ref{ie10})$, after some algebra, we eventually find that
\begin{equation}\label{ie11}
f(\eta)  =   \frac  { F_Q\, \Gamma((d-1)/2)} {2\, \pi^{(d+1)/2}}  \, \bigg\{  \frac {d-1}{2}\,  \Beta\bigg(\frac {\cos\eta - \cos\alpha} {1 + \cos\eta}; \frac{1}{2}, \frac{d}{2} \bigg)  +\bigg(\frac{1+\cos\alpha} {1+\cos\eta}\bigg)^{\frac{d-1}{2}} \sqrt{\frac{1+\cos\alpha}{\cos\eta-\cos\alpha}} \bigg\} + F(\eta),
\end{equation}
where $ 0  \leq \eta \leq \alpha$.

The expression in braces on the right hand side of $(\ref{ie11})$ represents (save for a normalizing constant) the equilibrium density of the spherical cap centered at the North Pole, for the case of no external field. Our goal at this stage will be to transform this expression into a form first obtained in \cite[p. 780, expression (44)]{bds1}, for the case of general Riesz potential. This will prove useful in our further considerations. 

Recalling that 
\begin{equation}\label{hypergeom-to-beta}
{}_2 F_1(1,a+b;a+1;z) = \frac {a} {z^a\, (1-z)^b} \, \Beta(z;a,b),
\end{equation}
and denoting for brevity $t:=\cos\alpha$, $u:=\cos\eta$, we see that 
\begin{align*}
 \Beta\bigg(\frac {\cos\eta - \cos\alpha} {1 + \cos\eta}; \frac{1}{2}, \frac{d}{2} \bigg) & =  \Beta\bigg(\frac {u - t} {1 + u}; \frac{1}{2}, \frac{d}{2} \bigg) \\
 & = \frac {\Gamma(1/2)} {\Gamma(3/2)} \bigg( \frac {u-t} {1+u} \bigg)^{1/2} \bigg( \frac {1+t} {1+u} \bigg)^{d/2} {}_2 F_1\bigg (1, \frac {d+1} {2};\frac {3} {2};  \frac {u-t} {1+u} \bigg) \\
 & = 2 \, \bigg( \frac {u-t} {1+u} \bigg)^{1/2} \bigg( \frac {1+t} {1+u} \bigg)^{d/2} {}_2 F_1\bigg (1, \frac {d+1} {2};\frac {3} {2};  \frac {u-t} {1+u} \bigg).
\end{align*}
Therefore,
\begin{align*}
 \frac {d-1}{2}\,  \Beta\bigg(\frac {\cos\eta - \cos\alpha} {1 + \cos\eta}; \frac{1}{2}, \frac{d}{2} \bigg) & + \bigg(\frac{1+\cos\alpha} {1+\cos\eta}\bigg)^{(d-1)/2} \sqrt{\frac{1+\cos\alpha}{\cos\eta-\cos\alpha}}  \\
& =  \frac {d-1}{2}\,  \Beta\bigg(\frac {u - t} {1 + u}; \frac{1}{2}, \frac{d}{2} \bigg) + \bigg(\frac{1+t} {1+u}\bigg)^{(d-1)/2} \bigg(\frac{1+t}{u-t} \bigg)^{1/2} \\
& = \bigg(\frac{1+t}{1+u} \bigg)^{(d-1)/2} \bigg(\frac{1+t}{u-t} \bigg)^{1/2} \, \bigg\{ 1 + (d-1) \, \frac {u-t} {1+u} \,  {}_2 F_1\bigg (1, \frac {d+1} {2};\frac {3} {2};  \frac {u-t} {1+u} \bigg) \bigg\}.
\end{align*}

\noindent For brevity, let $z:=(u-t)/(1+u)$. We will be working with the term $z\, {}_2 F_1 (1, (d+1)/ 2; 3/2; z) = z\, {}_2 F_1 (1, (d-1)/ 2+1; 1/2+1; z)$, appearing in the right hand side of the last expression. According to \cite[p. 558, \# 15.2.20]{abram},
\[
c\, (1-z)\,  {}_2 F_1 (a,b;c;z) - c \,  {}_2 F_1 (a-1,b;c;z) + (c-b)z \,  {}_2 F_1 (a,b;c+1;z) = 0,  
\]
which, in cojunction with the fact $ {}_2 F_1 (0,b;c;z)=1$, entails
\begin{equation}\label{hyp1}
z\, {}_2 F_1\bigg (1, \frac {d-1} {2}+1;\frac {1} {2}+1; z \bigg) = \frac {1}{d} \bigg\{ (1-z)\, {}_2 F_1\bigg (1, \frac {d-1} {2}+1;\frac {1} {2};  z \bigg) - 1 \bigg\}.
\end{equation}

\noindent Furthermore, \cite[p. 558, \# 15.2.14]{abram} states that
\[
 (b-a)\,  {}_2 F_1 (a,b;c;z) + a \,  {}_2 F_1 (a+1,b;c;z) - b \,  {}_2 F_1 (a,b+1;c;z) = 0,  
\]
which in our circumstances is equivalent to
\begin{equation}\label{hyp2}
{}_2 F_1\bigg (1, \frac {d-1} {2}+1;\frac {1} {2}; z \bigg) = \frac {1} {d-1} \bigg\{ (d-3)\, {}_2 F_1\bigg (1, \frac {d-1} {2};\frac {1} {2}; z \bigg) +2 \,  {}_2 F_1\bigg (2, \frac {d-1} {2};\frac {1} {2}; z \bigg) \bigg\}.
\end{equation}

\noindent Inserting $(\ref{hyp2})$ into $(\ref{hyp1})$, we easily obtain that
\begin{empheq}{align}\label{hyp3}
z\, {}_2 F_1\bigg (1, \frac {d-1} {2}+1;\frac {1} {2}+1; z \bigg) =  \frac {1}{d(d-1)} \bigg\{ & (d-3) (1-z)\, {}_2 F_1\bigg (1, \frac {d-1} {2};\frac {1} {2};  z \bigg) \\ \nonumber 
 & + 2 (1-z) \,  {}_2 F_1\bigg (2, \frac {d-1} {2};\frac {1} {2};  z \bigg) - (d-1) \bigg\}.
\end{empheq}

\noindent It can be shown, for example, by MATHEMATICA, that 
\begin{equation}\label{hyp4}
{}_2 F_1 (2, b;c; z) = \frac {1} {z-1} \bigg\{ (c - 2 + z (1-b)) \, {}_2 F_1 (1, b;c; z) + 1 - c \bigg\}.
\end{equation}

\noindent Substituting $(\ref{hyp4})$ into $(\ref{hyp3})$ and simplifying, we deduce that
\begin{equation*}
z\, {}_2 F_1\bigg (1, \frac {d-1} {2}+1;\frac {1} {2}+1; z \bigg) =  \frac {1}{d-1} \bigg\{ {}_2 F_1\bigg (1, \frac {d-1} {2};\frac {1} {2};  z \bigg) - 1 \bigg\}.
\end{equation*}

\noindent We thus demonstrated that 
\begin{align*}
 \frac {d-1}{2}\,  \Beta\bigg(\frac {u - t} {1 + u}; \frac{1}{2}, \frac{d}{2} \bigg) & + \bigg(\frac{1+t} {1+u}\bigg)^{(d-1)/2} \bigg(\frac{1+t}{u-t} \bigg)^{1/2} \\ 
& = \bigg(\frac{1+t}{1+u} \bigg)^{(d-1)/2} \bigg(\frac{1+t}{u-t} \bigg)^{1/2} \, {}_2 F_1\bigg (1, \frac {d-1} {2};\frac {1} {2};  \frac{u-t}{1+u} \bigg)  \\ 
& = \bigg(\frac{1+\cos\alpha}{1+\cos\eta} \bigg)^{(d-1)/2} \bigg(\frac{1+\cos\alpha}{\cos\eta-\cos\alpha} \bigg)^{1/2} \, {}_2 F_1\bigg (1, \frac {d-1} {2};\frac {1} {2};  \frac{\cos\eta-\cos\alpha}{1+\cos\eta} \bigg)
\end{align*}

\noindent Expression $(\ref{ie11})$ can now be written as 
\begin{equation}\label{ie12}
f(\eta)  =    \frac  { F_Q\, \Gamma((d-1)/2)} {2\, \pi^{(d+1)/2}} \, \bigg(\frac{1+\cos\alpha}{1+\cos\eta} \bigg)^{\frac{d-1}{2}}  \bigg(\frac{1+\cos\alpha}{\cos\eta-\cos\alpha} \bigg)^{\frac{1}{2}} \, {}_2 F_1\bigg (1, \frac {d-1} {2};\frac {1} {2};  \frac{\cos\eta-\cos\alpha}{1+\cos\eta} \bigg) + F(\eta),
\end{equation}
where $ 0 \leq \eta \leq \alpha$.
% -- calculation of the Robin constant

Next, we compute the Robin constant $F_Q$. Recall that $\mu_Q$ is a probability measure, so that its mass is $1$. Therefore,
\begin{equation}\label{mmo}
1=\int d\mu_Q = \int_0^\alpha \int_{\S^{d-2}} f(\eta)\, \sin^{d-2}\eta \, d\sigma_{d-1}\, d\eta.
\end{equation}
Inserting expression $(\ref{ie12})$ into $(\ref{mmo})$, we obtain
\begin{empheq}{align}\label{constfq}
& \frac{\Gamma((d-1)/2)}{2\pi^{(d-1)/2}}  = \frac{1}{\omega_{d-1}}    =  \int_0^\alpha f(\eta) \, \sin^{d-2}\eta \, d\eta \\ \nonumber
	& = \int_0^\alpha F(\eta) \, \sin^{d-2}\eta \, d\eta \\ \nonumber
	& + \frac  { F_Q\, \Gamma((d-1)/2)} {2\, \pi^{(d+1)/2}}  \, \int_0^\alpha \bigg(\frac{1+\cos\alpha}{1+\cos\eta} \bigg)^{\frac{d-1}{2}} \bigg(\frac{1+\cos\alpha}{\cos\eta-\cos\alpha} \bigg)^{\frac{1}{2}} \, {}_2 F_1\bigg (1, \frac {d-1} {2};\frac {1} {2};  \frac{\cos\eta-\cos\alpha}{1+\cos\eta} \bigg) \, \sin^{d-2}\eta \, d\eta.
\end{empheq}

\noindent To evaluate the second integral in the right hand side of the last expression, we will be making use of the following result \cite[Lemma A.1, p. 40]{bdsArxiv}.
\begin{lemma}\label{lemmabds}
Assume $-1 \leq a < b < c \leq 1$ and $|y|\leq1$. Set $x:=(b-a)/(c-a)$. Then for all $\alpha, \beta, \gamma >0$ such that $\beta + \gamma > \alpha$, one has
\begin{empheq}{align*}
 \int_a^b (u-a)^{\beta-1}\, (b-u)^{\gamma-1}\, (c-u)^{-\alpha} \, & {}_2 F_1\bigg (\alpha, \beta; \gamma; y\, \frac{b-u}{c-u} \bigg) \,du \\ 
& = \frac {\Gamma(\beta) \, \Gamma(\gamma)} {\Gamma(\beta+\gamma-\alpha) \, \Gamma(\alpha)} \, (b-a)^{\beta+\gamma-1} \, (c-a)^{-\gamma} \, (c-b)^{\gamma - \alpha} \, (1-xy)^{-\beta} \\ 
& \times  \int_0^1 t^{\beta+\gamma-\alpha-1} \, (1-t)^{\alpha-1}\, (1-xt)^{\beta-\gamma} \, \bigg( 1 - \frac {x(1-y)} {1-xy}\, t  \bigg)^{-\beta} \, dt.
\end{empheq}
\end{lemma}

\noindent We first bring the second integral in the right hand side of $(\ref{constfq})$ into a form that can be handled by invoking Lemma \ref{lemmabds}. Using trivial substitutions, we write

\begin{align*}
 \int_0^\alpha  \bigg(\frac{1+\cos\alpha}{1+\cos\eta} \bigg)^{\frac{d-1}{2}} \, & \bigg(\frac{1+\cos\alpha}{\cos\eta-\cos\alpha} \bigg)^{\frac{1}{2}} \, {}_2 F_1\bigg (1, \frac {d-1} {2};\frac {1} {2};  \frac{\cos\eta-\cos\alpha}{1+\cos\eta} \bigg) \, \sin^{d-2}\eta \, d\eta \\
& = (1+t^*)^{d/2} \, \int_{-1}^{t^*} (t^*-u)^{-1/2} \, (1-u)^{-1} \, (1+u)^{(d-3)/2} \, {}_2 F_1\bigg (1, \frac {d-1} {2};\frac {1} {2};  \frac{t^*-u}{1-u} \bigg) \, du,
\end{align*}
where we set $t^*:=-\cos\alpha$. 

Applying Lemma \ref{lemmabds} with $a=-1, b=t^*, c=1, y=1 $ and $\alpha=1, \beta=(d-1)/2, \gamma = 1/2$, while reasoning along the lines of the proof of Lemma 30 in \cite[p. 782]{bds1}, we find that
\begin{align*}
 (1+t)^{d/2} \, \int_{-1}^{t^*} (t^*-u)^{-1/2} \, (1-u)^{-1} \, (1+u)^{(d-3)/2} \, & {}_2 F_1\bigg (1, \frac {d-1} {2};\frac {1} {2};  \frac{t^*-u}{1-u} \bigg) \, du \\
& =  \frac {\sqrt{\pi} \, 2^{d-2} \, \Gamma{((d-1)/2)} } {\Gamma(d/2-1)} \, \Beta\bigg(\sin^2\bigg(\frac {\alpha} {2} \bigg); \frac{d-2}{2}, \frac{d}{2} \bigg).
\end{align*}
Thus we conclude that 
\begin{empheq}{align}\label{robint1}
 \int_0^\alpha  \bigg(\frac{1+\cos\alpha}{1+\cos\eta} \bigg)^{\frac{d-1}{2}} \, \bigg(\frac{1+\cos\alpha}{\cos\eta-\cos\alpha} \bigg)^{\frac{1}{2}} \,   {}_2 F_1\bigg (1, & \frac {d-1} {2};\frac {1} {2};  \frac{\cos\eta-\cos\alpha}{1+\cos\eta} \bigg) \, \sin^{d-2}\eta \, d\eta \\ \nonumber 
 & = \frac {\sqrt{\pi} \, 2^{d-2} \, \Gamma{((d-1)/2)} } {\Gamma(d/2-1)} \, \Beta\bigg(\sin^2\bigg(\frac {\alpha} {2} \bigg); \frac{d-2}{2}, \frac{d}{2} \bigg).
\end{empheq}

\noindent Substituting $(\ref{robint1})$ into $(\ref{constfq})$, it follows that 
\begin{equation}\label{robinconstNP}
 F_Q  = \frac {\pi^{d/2} \, \Gamma(d/2-1)} {2^{d-3}\, (\Gamma((d-1)/2))^2 } \, \bigg(\Beta\bigg(\sin^2\bigg(\frac {\alpha} {2} \bigg); \frac{d-2}{2}, \frac{d}{2} \bigg) \bigg)^{-1} \,  \bigg\{  \frac{\Gamma((d-1)/2)}{2\pi^{(d-1)/2}} -  \int_0^\alpha F(\eta) \, \sin^{d-2}\eta \, d\eta \bigg\}.
\end{equation}
In light of $(\ref{robinconstNP})$, the expression for the equilibrium density $(\ref{ie12})$ can be written as
\begin{equation*}
f(\eta)  =   C_Q  \bigg(\frac{1+\cos\alpha}{1+\cos\eta} \bigg)^{\frac{d-1}{2}}  \bigg(\frac{1+\cos\alpha}{\cos\eta-\cos\alpha} \bigg)^{\frac{1}{2}} \, {}_2 F_1\bigg (1, \frac {d-1} {2};\frac {1} {2};  \frac{\cos\eta-\cos\alpha}{1+\cos\eta} \bigg) + F(\eta), \quad 0 \leq \eta \leq \alpha,
\end{equation*}
with the constant $C_Q$ given by
\begin{equation*}
C_Q  =  \frac {\Gamma(d/2-1)} {2^{d-2}\, \sqrt{\pi}\, \Gamma((d-1)/2) } \, \bigg(\Beta\bigg(\sin^2\bigg(\frac {\alpha} {2} \bigg); \frac{d-2}{2}, \frac{d}{2} \bigg) \bigg)^{-1} \,  \bigg\{  \frac{\Gamma((d-1)/2)}{2\pi^{(d-1)/2}} -  \int_0^\alpha F(\eta) \, \sin^{d-2}\eta \, d\eta \bigg\}.
\end{equation*}
\qed
% -------------------- Theorem 1.4 ---------------------------------- Theorem 1.4 ------------------------- Theorem 1.4 ---------------------------------- Theorem 1.4 ------------------------

\noindent{\bf Proof of Theorem \ref{theo3}.} If $S_Q=C_{S,\alpha}$, equation $(\ref{ieec})$ assumes the form
\begin{equation}\label{ie21}
\frac{2\pi^{(d-2)/2}} {\Gamma((d-2)/2)}  \int_\alpha^\pi  f(\eta) \sin^{d-2}\eta \, d\eta \, \int_0^\pi \frac {\sin^{d-3}\xi \,d\xi} {(2 - 2\gamma)^{(d-2)/2}} = F_Q - Q(\theta), \quad \alpha \leq \theta \leq \pi,
\end{equation}
where $\gamma = \cos\theta \cos\eta + \sin\theta \sin\eta \cos\xi$. Via the change of variables $\widetilde{\theta} = \pi - \theta$, we transform $(\ref{ie21})$ into 
\begin{equation}\label{ie22}
\frac{2\pi^{(d-2)/2}} {\Gamma((d-2)/2)}  \int_0^\beta  f_0(\eta) \sin^{d-2} \eta \, d\eta \, \int_0^\pi \frac { \sin^{d-3} \xi \,d\xi} {(2 - 2\widetilde{\gamma})^{(d-2)/2}} = F_Q - Q_0(\widetilde{\theta}), \quad 0 \leq \widetilde{\theta} \leq \beta,
\end{equation}
with $\beta=\pi-\alpha$, $f_0(\widetilde{\eta})=f(\pi-\widetilde{\eta})$, $Q_0(\widetilde{\theta})=Q(\pi-\widetilde{\theta})$, and $\widetilde{\gamma} = \cos \widetilde{\theta} \cos \eta + \sin \widetilde{\theta} \sin\eta \cos\xi$. The integral equation $(\ref{ie22})$ is of the form $(\ref{ie1})$. Hence Theorem \ref{theo2} applies, and we obtain
\begin{equation*}
F_0(\widetilde{\eta}) =  \frac  { \Gamma((d-2)/2)} {2\, \pi^{(d+2)/2}}  \frac {1} {\sin \widetilde{\eta} } \,\sec^{d-3}\bigg(\frac{\widetilde{\eta}}{2}\bigg) \, \frac {d} {d\widetilde{\eta}} \int_{\widetilde{\eta}}^\beta \frac {g_0(\zeta) \sin \zeta \, d\zeta} {\sqrt{\cos \widetilde{\eta} - \cos \zeta}}, \quad 0  \leq \widetilde{\eta} \leq \beta,
\end{equation*}
where
\begin{equation*}
g_0(\zeta) = \cot^{d-3}\bigg(\frac{\zeta}{2}\bigg)\, \frac {d} {d\zeta} \int_0^\zeta \frac {Q_0(\theta)\,  \sin^{d-3}(\theta/2) \sin\theta \, d\theta} {\sqrt{\cos\theta - \cos \zeta}}, \quad 0 \leq \zeta \leq \beta.
\end{equation*}
The density $f_0$ of the equilibrium measure $\mu_{Q_0}$ is
\begin{equation*}
f_0(\widetilde{\eta})  =   C_Q  \bigg(\frac{1+\cos\beta}{1+\cos\widetilde{\eta}} \bigg)^{\frac{d-1}{2}}  \bigg(\frac{1+\cos\beta}{\cos\widetilde{\eta}-\cos\beta} \bigg)^{\frac{1}{2}} \, {}_2 F_1\bigg (1, \frac {d-1} {2};\frac {1} {2};  \frac{\cos\widetilde{\eta}-\cos\beta}{1+\cos\widetilde{\eta}} \bigg) + F_0(\widetilde{\eta}), \quad 0  \leq \widetilde{\eta} \leq \beta,
\end{equation*}
where the constant $C_Q$ is given by
\begin{equation*}
C_Q  = \frac {\Gamma(d/2-1)} {2^{d-2}\, \sqrt{\pi} \, \Gamma((d-1)/2) } \, \bigg(\Beta\bigg(\sin^2\bigg(\frac {\beta} {2} \bigg); \frac{d-2}{2}, \frac{d}{2} \bigg) \bigg)^{-1} \,  \bigg\{  \frac{\Gamma((d-1)/2)}{2\pi^{(d-1)/2}} -  \int_0^\beta F_0(\widetilde{\eta}) \, \sin^{d-2}\widetilde{\eta} \, d\widetilde{\eta} \bigg\}.
\end{equation*}
Going back to the $\eta$ variable via $\eta= \pi - \widetilde{\eta}$, after some algebra, we find
\begin{equation*}
g(\zeta) = \tan^{d-3}\bigg(\frac{\zeta}{2}\bigg)\, \frac {d} {d\zeta} \int_\zeta^\pi \frac {Q(\theta)\,  \cos^{d-3}(\theta/2) \sin\theta \, d\theta} {\sqrt{\cos \zeta-\cos\theta}}, \quad \alpha \leq \zeta \leq \pi,
\end{equation*}
so that
\begin{equation*}
F(\eta) = \frac {\Gamma((d-2)/2)} {2\pi^{(d+2)/2}} \frac {1} {\sin\eta } \,\csc^{d-3}\bigg(\frac{\eta}{2}\bigg) \, \frac {d} {d\eta} \int_\alpha^{\eta}\frac {g(\zeta) \sin \zeta \, d\zeta} {\sqrt{\cos\zeta - \cos \eta}}, \quad \alpha \leq \eta \leq \pi.
\end{equation*}
The constant $C_Q$ has the form
\begin{equation*}
C_Q  = \frac {\Gamma(d/2-1)} {2^{d-2}\, \sqrt{\pi}\, \Gamma((d-1)/2) } \, \bigg(\Beta\bigg(\cos^2\bigg(\frac {\alpha} {2} \bigg); \frac{d-2}{2}, \frac{d}{2} \bigg) \bigg)^{-1} \,  \bigg\{  \frac{\Gamma((d-1)/2)}{2\pi^{(d-1)/2}} -  \int_\alpha^\pi F(\eta) \, \sin^{d-2}\eta \, d\eta \bigg\}.
\end{equation*}
We thus conclude that the equilibrium density, when support is a spherical cap centered at the South Pole, is given by
\begin{equation*}
f(\eta)  =   C_Q  \bigg(\frac{1-\cos\alpha}{1-\cos\eta} \bigg)^{\frac{d-1}{2}}  \bigg(\frac{1-\cos\alpha}{\cos\alpha-\cos\eta} \bigg)^{\frac{1}{2}} \, {}_2 F_1\bigg (1, \frac {d-1} {2};\frac {1} {2};  \frac{\cos\alpha-\cos\eta}{1-\cos\eta} \bigg) + F(\eta), \quad \alpha \leq \eta \leq \pi.
\end{equation*}
\qed

% ------------------------------------------------------------------------------- Theorem 2.4 -----------------------------------------------------------------------------------

\noindent{\bf Proof of Theorem \ref{msfptch}.} Recall that the external field $Q$ in question is given by $(\ref{extfieldexam})$, that is
\[
Q(\eta) = q\,(1-\cos\eta)^{-(d-2)/2}, \quad q>0, \quad 0 \leq \eta \leq \pi.
\]
Substituting this expression into formula $(\ref{msfunc})$ for the $\cF$-functional, we are lead to the following integral
\begin{equation*}
\int_\alpha^\pi  \bigg(\frac{1-\cos\alpha}{1-\cos\eta} \bigg)^{\frac{d-1}{2}}  \bigg(\frac{1-\cos\alpha}{\cos\alpha-\cos\eta} \bigg)^{\frac{1}{2}} {}_2 F_1\bigg (1, \frac {d-1} {2};\frac {1} {2};  \frac{\cos\alpha-\cos\eta}{1-\cos\eta} \bigg) \, (1-\cos\eta)^{-\frac{(d-2)}{2} } \sin^{d-2}\eta\, d\eta.
\end{equation*}
Letting $t:=\cos\alpha$, $u:=\cos\eta$, after some simple algebra, we obtain that
\begin{empheq}{align}\label{msfuncex_int2}
& \int_\alpha^\pi  \bigg(\frac{1-\cos\alpha}{1-\cos\eta} \bigg)^{\frac{d-1}{2}}  \bigg(\frac{1-\cos\alpha}{\cos\alpha-\cos\eta} \bigg)^{\frac{1}{2}} {}_2 F_1\bigg (1, \frac {d-1} {2};\frac {1} {2};  \frac{\cos\alpha-\cos\eta}{1-\cos\eta} \bigg) \, (1-\cos\eta)^{-\frac{(d-2)}{2} } \sin^{d-2}\eta\, d\eta \\ \nonumber
&= (1-t)^{d/2} \, \int_{-1}^t (1-u)^{-d/2} \, (1+u)^{(d-3)/2} \, (t-u)^{-1/2} \, {}_2 F_1\bigg (1, \frac {d-1} {2};\frac {1} {2};  \frac{t-u}{1-u} \bigg) \, du.
\end{empheq}
Hence, our original integral further reduces to 
\begin{equation}\label{msfuncex_int3}
\int_{-1}^t (1-u)^{-d/2} \, (1+u)^{(d-3)/2} \, (t-u)^{-1/2} \, {}_2 F_1\bigg (1, \frac {d-1} {2};\frac {1} {2};  \frac{t-u}{1-u} \bigg) \, du.
\end{equation}
The integral in $(\ref{msfuncex_int3})$ closely resembles the integral appearing in Lemma \ref{lemmabds}. To evaluate integral $(\ref{msfuncex_int3})$, we will develop an argument similar to the proof of Lemma A.1 in \cite{bdsArxiv}. Set $a=-1$, $b=t$, $c=1$, $\gamma=1/2$ and $\beta=(d-1)/2$. Also, we let $x=(b-a)/(c-a)$, so that $0<x<1$. Introducing the substitution $(b-u)/(c-u)=xv$, after simplifications we deduce that
\begin{empheq}{align}\label{msfuncex_int4}
& \int_{-1}^t (1-u)^{-d/2} \, (1+u)^{(d-3)/2} \, (t-u)^{-1/2} \, {}_2 F_1\bigg (1, \frac {d-1} {2};\frac {1} {2};  \frac{t-u}{1-u} \bigg) \, du \\ \nonumber
& = (b-a)^{\gamma+\beta-1}\, (c-a)^{-\gamma} \, (c-b)^{\gamma - d/2}\, \int_0^1 v^{\gamma-1} \, (1-v)^{\beta-1} \,  {}_2 F_1 (1, \beta;\gamma; xv ) \, dv \\ \nonumber
& = (b-a)^{\gamma+\beta-1}\, (c-a)^{-\gamma} \, (c-b)^{\gamma - d/2}\, I,
\end{empheq}
where
\begin{equation}\label{msfuncex_integI}
I:= \int_0^1 v^{\gamma-1} \, (1-v)^{\beta-1} \,  {}_2 F_1 (1, \beta;\gamma; xv ) \, dv.
\end{equation}
Substituting the series expansion for the Gauss hypergeometric function $(\ref{gausshyperdef})$ into the above integral and integrating term-by-term, we find
\begin{align*}
I & = \int_0^1 v^{\gamma-1} \, (1-v)^{\beta-1} \,  {}_2 F_1 (1, \beta;\gamma; xv ) \, dv \\
& =  \int_0^1 v^{\gamma-1} \, (1-v)^{\beta-1} \,  dv \,  \sum_{n=0}^\infty \frac {(1)_n\, (\beta)_n} {(\gamma)_n \, n!} \, x^n \, v^n \\
& = \sum_{n=0}^\infty \frac {(\beta)_n} {(\gamma)_n} \, x^n \, \int_0^1 v^{n + \gamma - 1} \, (1-v)^{\beta-1} \,  dv \\
& =  \sum_{n=0}^\infty \frac {(\beta)_n} {(\gamma)_n} \, x^n \, \Beta(n+\gamma,\beta) \\
& =  \sum_{n=0}^\infty \frac {(\beta)_n} {(\gamma)_n} \frac {\Gamma(n+\gamma)\, \Gamma(\beta)} {\Gamma(n+\gamma+\beta)} \, x^n.
\end{align*}
Taking into account that $(a)_n = \Gamma(a+n)/\Gamma(a)$, we further obtain
\begin{empheq}{align}\label{msfuncex_int5}
I & = \sum_{n=0}^\infty \frac {(\beta)_n} {(\gamma)_n} \frac {\Gamma(n+\gamma)\, \Gamma(\beta)} {\Gamma(n+\gamma+\beta)} \, x^n \\ \nonumber
	& = \Gamma(\gamma) \,  \sum_{n=0}^\infty \frac {\Gamma(\beta+n)}{\Gamma(\beta+\gamma+n)} \, x^n \\ \nonumber
	& = \Gamma(1/2) \,  \sum_{n=0}^\infty \frac {\Gamma(n+(d-1)/2)}{\Gamma(n+d/2)} \, x^n \\ \nonumber
	& = \frac {\sqrt{\pi}\, \Gamma((d-1)/2)} {\Gamma(d/2)} \, {}_2 F_1\bigg (1, \frac {d-1} {2};\frac {d} {2};  x \bigg).
\end{empheq}
Substituting $(\ref{msfuncex_int5})$ into $(\ref{msfuncex_int4})$, and simplifying, we find that
\begin{empheq}{align*}
& \int_{-1}^t (1-u)^{-d/2} \, (1+u)^{(d-3)/2} \, (t-u)^{-1/2} \, {}_2 F_1\bigg (1, \frac {d-1} {2};\frac {1} {2};  \frac{t-u}{1-u} \bigg) \, du \\ 
& = \frac {\sqrt{\pi}\, \Gamma((d-1)/2)} {\sqrt{2} \, \Gamma(d/2)} \, (1-\cos\alpha)^{-(d-1)/2} \, (1+\cos\alpha)^{(d-2)/2} \, {}_2 F_1\bigg (1, \frac {d-1} {2};\frac {d} {2};  \frac{1+\cos\alpha}{2} \bigg).
\end{empheq}
Inserting the last integral into $(\ref{msfuncex_int2})$, we finally infer that
\begin{empheq}{align}\label{msfuncex_int7}
 \int_\alpha^\pi  Q(\eta) & \, \bigg(\frac{1-\cos\alpha}{1-\cos\eta} \bigg)^{\frac{d-1}{2}}  \bigg(\frac{1-\cos\alpha}{\cos\alpha-\cos\eta} \bigg)^{\frac{1}{2}} {}_2 F_1\bigg (1, \frac {d-1} {2};\frac {1} {2};  \frac{\cos\alpha-\cos\eta}{1-\cos\eta} \bigg) \, \sin^{d-2}\eta\, d\eta \\ \nonumber
& =  \frac {q\, \sqrt{\pi}\, \Gamma((d-1)/2)} {\sqrt{2} \, \Gamma(d/2)} \, (1-\cos\alpha)^{1/2} \, (1+\cos\alpha)^{(d-2)/2} \, {}_2 F_1\bigg (1, \frac {d-1} {2};\frac {d} {2};  \frac{1+\cos\alpha}{2} \bigg).
\end{empheq}
Furthermore, from $(\ref{hypergeom-to-beta})$ it follows that
\begin{equation}\label{ex1transform}
(1-\cos\alpha)^{1/2} \, (1+\cos\alpha)^{(d-2)/2} \, {}_2 F_1\bigg (1, \frac {d-1} {2};\frac {d} {2};  \frac{1+\cos\alpha}{2} \bigg) = \frac {2^{(d-1)/2} \, \Gamma(d/2)} {\Gamma((d-2)/2)} \, \Beta \bigg( \cos^2\bigg(\frac {\alpha} {2} \bigg); \frac{d-2}{2},\frac{1}{2}\bigg),
\end{equation}
thus reducing expression $(\ref{msfuncex_int7})$ to
\begin{empheq}{align}\label{msfuncex_int8}
 \int_\alpha^\pi  Q(\eta)  \, \bigg(\frac{1-\cos\alpha}{1-\cos\eta} \bigg)^{\frac{d-1}{2}} \bigg(\frac{1-\cos\alpha}{\cos\alpha-\cos\eta} \bigg)^{\frac{1}{2}} &  {}_2 F_1\bigg (1,  \frac {d-1} {2};\frac {1} {2};   \frac{\cos\alpha-\cos\eta}{1-\cos\eta} \bigg) \, \sin^{d-2}\eta\, d\eta \\ \nonumber
& =  \frac {q\, \sqrt{\pi}\, 2^{(d-2)/2} \, \Gamma((d-1)/2)} {\Gamma((d-2)/2)} \, \Beta \bigg( \cos^2\bigg(\frac {\alpha} {2} \bigg) ; \frac{d-2}{2},\frac{1}{2}\bigg).
\end{empheq}
Substituting $(\ref{msfuncex_int8})$ into $(\ref{msfunc})$, we finally obtain the desired expression $(\ref{msf-p-ch})$ for the $\cF$-functional,
\begin{empheq}{align}\label{msfex1}
 \cF(C_{S,\alpha})  =  \frac {\sqrt{\pi}\,  \Gamma(d/2-1)} {2^{d-2}\,  \Gamma((d-1)/2)} \, \bigg(\Beta\bigg(\cos^2\bigg(\frac {\alpha} {2} \bigg); & \frac{d-2}{2},  \frac{d}{2} \bigg) \bigg)^{-1}  \times \\ \nonumber
 & \bigg\{ 1 + \frac {q\,  2^{(d-2)/2} \, \Gamma((d-1)/2)} {\sqrt{\pi} \, \Gamma(d/2-1)} \, \Beta \bigg( \cos^2\bigg(\frac {\alpha} {2} \bigg) ; \frac{d-2}{2},\frac{1}{2}\bigg) \bigg\}.
\end{empheq}
\qed

% ------------------------------------------------------------ Example ----------------------------------------------------------------------------------------------------------------------------

\noindent{\bf Proof of Theorem \ref{densexamtheo}.} Assume that the support is a spherical cap $C_{S,\alpha}$, and an external field $Q$ on the sphere $\S^{d-1}$ is given by $(\ref{extfieldexam})$, that is
\[
Q(\theta) = q\,(1-\cos\theta)^{-(d-2)/2}, \quad q>0, \quad \alpha \leq \theta \leq \pi.
\]
The $\cF$-functional for this external field is given by expression $(\ref{msfex1})$. Taking into account that 
\begin{equation}\label{beta-func-diff}
\frac{d}{dz} \, \Beta(z;a,b) = (1-z)^{b-1}\, z^{a-1},
\end{equation}
and differentiating $(\ref{msfex1})$ with respect to $\alpha$, one can show that
\begin{empheq}{align*}
&  \cF'(C_{S,\alpha}) = \frac {\sqrt{\pi}\,  \Gamma(d/2-1)} {2^{d-2}\,  \Gamma((d-1)/2)} \, \bigg(\Beta\bigg(\cos^2\bigg(\frac {\alpha} {2} \bigg);  \frac{d-2}{2},  \frac{d}{2} \bigg) \bigg)^{-2} \, \cos^{d-3}\bigg(\frac {\alpha} {2} \bigg) \, \sin^{d-1}\bigg(\frac {\alpha} {2} \bigg) \times  \\ 
& \bigg\{1 - \frac{q\,2^{(d-2)/2}\, \Gamma((d-1)/2)} {\sqrt{\pi}\,  \Gamma(d/2-1)} \bigg[ \csc^{d-1}\bigg(\frac {\alpha} {2} \bigg) \, \Beta\bigg(\cos^2\bigg(\frac {\alpha} {2} \bigg);  \frac{d-2}{2},  \frac{d}{2} \bigg) - \Beta\bigg(\cos^2\bigg(\frac {\alpha} {2} \bigg);  \frac{d-2}{2},  \frac{1}{2} \bigg) \bigg] \bigg\}.
\end{empheq}
This shows that the critical points of $\cF(C_{S,\alpha})$ satisfy
\begin{equation}\label{msfex1eqn-crit-pnts}
 \csc^{d-1}\bigg(\frac {\alpha} {2} \bigg) \, \Beta\bigg(\cos^2\bigg(\frac {\alpha} {2} \bigg);  \frac{d-2}{2},  \frac{d}{2} \bigg) - \Beta\bigg(\cos^2\bigg(\frac {\alpha} {2} \bigg);  \frac{d-2}{2},  \frac{1}{2} \bigg) = \frac  {\sqrt{\pi}\,  \Gamma(d/2-1)} {q\,2^{(d-2)/2}\, \Gamma((d-1)/2)}.
\end{equation}
The proof of existence and uniqueness of a zero for the function defined by expression $(\ref{msfex1eqn-crit-pnts})$ is exactly the same in \cite[Theorem 13, p. 773]{bds1}. We therefore obtained an equation for finding an angle $\alpha$ that defines the support  $C_{S,\alpha}$ of the extremal measure $\mu_Q$, when the external field is produced by a positive point charge of magnitude $q$, placed at the North Pole of the sphere $\S^{d-1}$. 

We next obtain the expression for the equilibrium density, corresponding to the external field under consideration. Applying Theorem \ref{theo3}, we first compute the auxiliary function $g(\zeta)$, according to $(\ref{auxsSP})$. We are thus led to the following integral, appearing in right hand side of $(\ref{auxsSP})$,
\[
\int_\zeta^\pi \frac {Q(\theta) \, \cos^{d-3}(\theta/2)\, \sin\theta \, d\theta} {\sqrt{\cos\zeta - \cos\theta}} = q\, 2^{-(d-3)/2} \, \int_\zeta^\pi \frac {(1-\cos\theta)^{-(d-2)/2} \, (1+\cos\theta)^{(d-3)/2}\, \sin\theta \, d\theta} {\sqrt{\cos\zeta - \cos\theta}}
\]
Using the substitution $1+\cos\theta=(1+\cos\zeta)\,t$, after a number of easy manipulations, we infer that
 \begin{empheq}{align*}
  \int_\zeta^\pi \frac {(1-\cos\theta)^{-(d-2)/2} \, (1+\cos\theta)^{(d-3)/2}\, \sin\theta \, d\theta} {\sqrt{\cos\zeta - \cos\theta}} & \\ 
  = \cos^{d-2}(\zeta/2)\,   \int_0^1 & t^{(d-3)/2} \, (1-t)^{-1/2} \, (1-\cos^2(\zeta/2) t)^{-(d-2)/2} \, dt \\ 
 =  \frac {\sqrt{\pi} \, \Gamma((d-1)/2)} {\Gamma(d/2)}&  \, \cos^{d-2}\bigg(\frac{\zeta}{2}\bigg) \, _2 F_1\bigg(\frac{d-2}{2},\frac{d-1}{2};\frac{d}{2};\cos^2\bigg(\frac{\zeta}{2}\bigg)\bigg),
\end{empheq}
where we used the integral representation of the hypergeometric function \cite[p. 558, \# 15.3.1]{abram}. Therefore,
\begin{equation*}
\int_\zeta^\pi \frac {Q(\theta) \, \cos^{d-3}(\theta/2)\, \sin\theta \, d\theta} {\sqrt{\cos\zeta - \cos\theta}} =  \frac {q\, \sqrt{\pi} \, \Gamma((d-1)/2)} {2^{(d-3)/2}\, \Gamma(d/2)}  \, \cos^{d-2}\bigg(\frac{\zeta}{2}\bigg) \, _2 F_1\bigg(\frac{d-2}{2},\frac{d-1}{2};\frac{d}{2};\cos^2\bigg(\frac{\zeta}{2}\bigg)\bigg).
\end{equation*}
Differentiating the last expression, and taking into account the fact \cite[p. 556, \#15.2.1]{abram}
\[
\frac {d}{dz}  \,_2 F_1(a,b;c;z) = \frac {ab}{c} \, _2 F_1(a+1,b+1;c+1;z), 
\]
we find, upon inserting the result of the differentiation into $(\ref{auxsSP})$,
\begin{empheq}{align}\label{functg}
g(\zeta) = -  \frac {q\, \sqrt{\pi}\, (d-2) \, \Gamma((d-1)/2)}  {2^{(d-1)/2}\, \Gamma(d/2)}  \, \sin^{d-2}\bigg(\frac{\zeta}{2}\bigg) \, \bigg\{ & {}_2 F_1\bigg(\frac{d-2}{2},\frac{d-1}{2};\frac{d}{2};\cos^2\bigg(\frac{\zeta}{2}\bigg)\bigg) \\\nonumber 
+  \frac{d-1}{d} \, \cos^2\bigg(\frac{\zeta}{2}\bigg) \,  &  {}_2 F_1\bigg(\frac{d}{2},\frac{d+1}{2};\frac{d}{2}+1;\cos^2\bigg(\frac{\zeta}{2}\bigg)\bigg) \bigg\}, \quad \alpha_0 \leq \zeta \leq \pi.
\end{empheq}
It turns out that the expression in braces on the right hand side of $(\ref{functg})$ can be simplified rather dramatically. Indeed, according to the linear transformation formula for the hypergeometric function \cite[p. 559, \#15.3.3]{abram}, we have
\[
{}_2 F_1(a,b;c;z) = (1-z)^{c-a-b} \, {}_2 F_1(c-a,c-b;c;z).
\]
With this in mind, it is a straightforward calculation to see that
\begin{equation}\label{hypergeom-transform-1}
{}_2 F_1\bigg(\frac{d-2}{2},\frac{d-1}{2};\frac{d}{2};\cos^2\bigg(\frac{\zeta}{2}\bigg)\bigg) = \sin^{-(d-3)} \bigg(\frac{\zeta}{2}\bigg) \, {}_2 F_1\bigg(1,\frac{1}{2};\frac{d}{2};\cos^2\bigg(\frac{\zeta}{2}\bigg)\bigg).
\end{equation}
Similarly,
\begin{equation}\label{hypergeom-transform-2}
 {}_2 F_1\bigg(\frac{d}{2},\frac{d+1}{2};\frac{d}{2}+1;\cos^2\bigg(\frac{\zeta}{2}\bigg)\bigg) = \sin^{-(d-1)} \bigg(\frac{\zeta}{2}\bigg) \, {}_2 F_1\bigg(1,\frac{1}{2};\frac{d}{2}+1;\cos^2\bigg(\frac{\zeta}{2}\bigg)\bigg).
\end{equation}
In light of $(\ref{hypergeom-transform-1})$ and $(\ref{hypergeom-transform-2})$, we see that 
\begin{empheq}{align}\label{hypergeom-transform-3}
 & {}_2 F_1\bigg(\frac{d-2}{2},\frac{d-1}{2};\frac{d}{2};\cos^2\bigg(\frac{\zeta}{2}\bigg)\bigg) +  \frac{d-1}{d} \, \cos^2\bigg(\frac{\zeta}{2}\bigg) \, {}_2 F_1\bigg(\frac{d}{2},\frac{d+1}{2};\frac{d}{2}+1;\cos^2\bigg(\frac{\zeta}{2}\bigg)\bigg) \\ \nonumber
& = \sin^{-(d-3)} \bigg(\frac{\zeta}{2}\bigg) \bigg\{ {}_2 F_1\bigg(1,\frac{1}{2};\frac{d}{2};\cos^2\bigg(\frac{\zeta}{2}\bigg)\bigg) + \frac{d-1}{d}
 \, \frac{\cos^2(\zeta/2)} {1-\cos^2(\zeta/2)} \, {}_2 F_1\bigg(1,\frac{1}{2};\frac{d}{2}+1;\cos^2\bigg(\frac{\zeta}{2}\bigg)\bigg) \bigg\}.
\end{empheq}
Using again the relation \cite[p. 558, \# 15.2.20]{abram},
\[
c\, (1-z)\,  {}_2 F_1 (a,b;c;z) - c \,  {}_2 F_1 (a-1,b;c;z) + (c-b)z \,  {}_2 F_1 (a,b;c+1;z) = 0,  
\]
it follows that
\[
 {}_2 F_1\bigg(1,\frac{1}{2};\frac{d}{2};\cos^2\bigg(\frac{\zeta}{2}\bigg)\bigg) + \frac{d-1}{d}
 \, \frac{\cos^2(\zeta/2)} {1-\cos^2(\zeta/2)} \, {}_2 F_1\bigg(1,\frac{1}{2};\frac{d}{2}+1;\cos^2\bigg(\frac{\zeta}{2}\bigg)\bigg) =  \sin^{-2}\bigg(\frac{\zeta}{2}\bigg),
\]
which, in conjunction with $(\ref{hypergeom-transform-3})$, allows us to conclude
\[
 {}_2 F_1\bigg(\frac{d-2}{2},\frac{d-1}{2};\frac{d}{2};\cos^2\bigg(\frac{\zeta}{2}\bigg)\bigg) +  \frac{d-1}{d} \, \cos^2\bigg(\frac{\zeta}{2}\bigg) \, {}_2 F_1\bigg(\frac{d}{2},\frac{d+1}{2};\frac{d}{2}+1;\cos^2\bigg(\frac{\zeta}{2}\bigg)\bigg) = \sin^{-(d-1)} \bigg(\frac{\zeta}{2}\bigg).
\]
Inserting the latter into $(\ref{functg})$, we obtain the desired simplified expression for $g(\zeta)$,
\begin{equation}\label{functg-simplf}
g(\zeta) = -  \frac {q\, \sqrt{\pi}\, (d-2) \, \Gamma((d-1)/2)}  {2^{(d-1)/2}\, \Gamma(d/2)}  \, \csc\bigg(\frac{\zeta}{2}\bigg), \quad \alpha_0 \leq \zeta \leq \pi.
\end{equation}
Having $g(\zeta)$ computed, we now proceed to evaluating the term $F(\eta)$ given by $(\ref{auxfSP})$, which describes the contribution of the external field $Q$ in the expression $(\ref{equildensSP})$ for the equilibrium measure $\mu_Q$. Upon inserting $(\ref{functg-simplf})$ into the right hand side of $(\ref{auxfSP})$, we are presented with evaluation of the following integral,
\[
\int_{\alpha_0}^\eta \frac {g(\zeta) \, \sin\zeta \, d\zeta} {\sqrt{\cos\zeta - \cos\eta}} =  -  \frac {q\, \sqrt{\pi}\, (d-2) \, \Gamma((d-1)/2)}  {2^{(d-2)/2}\, \Gamma(d/2)}  \,  \int_{\alpha_0}^\eta \frac {\sin\zeta \, d\zeta} {\sqrt{1-\cos\zeta} \, \sqrt{\cos\zeta - \cos\eta}}.
\]
The integral on the right hand side of the last expression can be easily evaluated using the substitution $t^2=\cos\zeta - \cos\eta$, so that
\[
\int_{\alpha_0}^\eta \frac {\sin\zeta \, d\zeta} {\sqrt{1-\cos\zeta} \, \sqrt{\cos\zeta - \cos\eta}} = 2 \, \sin^{-1}\sqrt{\frac{\cos\alpha_0-\cos\eta}{1-\cos\eta}}.
\]
Now it is not hard to see that
\begin{equation}\label{functg-simplf-diff}
\frac{d}{d\eta} \, \int_{\alpha_0}^\eta \frac {g(\zeta) \, \sin\zeta \, d\zeta} {\sqrt{\cos\zeta - \cos\eta}} =  - \frac {q\, \sqrt{\pi}\, (d-2) \, \Gamma((d-1)/2)}  {2^{(d-2)/2}\, \Gamma(d/2)}  \, \frac {\sin\eta} {1-\cos\eta} \, \sqrt{\frac{1-\cos\alpha_0}{\cos\alpha_0-\cos\eta}}, \quad \alpha_0\leq\eta\leq\pi.
\end{equation}
Inserting $(\ref{functg-simplf-diff})$ into $(\ref{auxfSP})$ and simplifying, we find
\begin{equation*}
F(\eta) = - \frac {q\, \Gamma((d-1)/2)}  {\sqrt{2}\,  \pi^{(d+1)/2}}  \, \frac {1} {(1-\cos\eta)^{(d-1)/2}} \, \sqrt{\frac{1-\cos\alpha_0}{\cos\alpha_0-\cos\eta}}, \quad \alpha_0\leq\eta\leq\pi.
\end{equation*}
Having $F(\eta)$ at hand, the constant $C_Q$ is found from $(\ref{constCQSP})$, thus leading us to the integral
\begin{equation}\label{F-1}
\int_{\alpha_0}^\pi F(\eta) \, \sin^{d-2}\eta \, d\eta = - \frac {q\, \Gamma((d-1)/2)}  {\sqrt{2}\,  \pi^{(d+1)/2}} \, \sqrt{1-\cos\alpha_0} \, \int_{\alpha_0}^\pi \frac{(1+\cos\eta)^{(d-3)/2} \, \sin\eta \, d\eta}{(1-\cos\eta) \, \sqrt{\cos\alpha_0 - \cos\eta}}.
\end{equation}
Using the standard substitution $1+\cos\eta=2t$, it follows that 
\begin{align*}
\int_{\alpha_0}^\pi \frac{(1+\cos\eta)^{(d-3)/2} \, \sin\eta \, d\eta}{(1-\cos\eta) \, \sqrt{\cos\alpha_0 - \cos\eta}} & = 2^{(d-4)/2} \, \bigg(\frac{1+\cos\alpha_0}{2}\bigg)^{(d-2)/2} \, \frac {\Gamma((d-1)/2)\, \sqrt{\pi}} {\Gamma(d/2)} \, {}_2 F_1\bigg(1,\frac{d-1}{2};\frac{d}{2};\cos^2\bigg(\frac{\alpha_0}{2}\bigg)\bigg) \\
& = \frac {\sqrt{\pi}\, 2^{(d-3)/2} \, \Gamma((d-1)/2)} {\Gamma(d/2-1)} \, \frac{1}{\sqrt{1-\cos\alpha_0}} \, \Beta\bigg(\cos^2\bigg(\frac{\alpha_0}{2}\bigg); \frac{d-2}{2}, \frac{1}{2}\bigg),
\end{align*}
where we used the integral representation for the hypergeometric function \cite[p. 558, \# 15.3.1]{abram}, as well as relation $(\ref{ex1transform})$. Substituting the value of the last integral into $(\ref{F-1})$, we find that
\begin{equation}\label{F-2}
\int_{\alpha_0}^\pi F(\eta) \, \sin^{d-2}\eta \, d\eta = - \frac {q\, 2^{(d-4)/2} \, (\Gamma((d-1)/2))^2}  {\pi^{d/2} \, \Gamma(d/2-1)} \, \Beta\bigg(\cos^2\bigg(\frac{\alpha_0}{2}\bigg); \frac{d-2}{2}, \frac{1}{2}\bigg).
\end{equation}
Hence, after substituting $(\ref{F-2})$ into $(\ref{constCQSP})$ and some simple algebra, we infer that the constant $C_Q$ is given by
\begin{empheq}{align*}
C_Q  = \frac {\Gamma(d/2-1)} {2^{d-1}\, \pi^{d/2}} \, \bigg(\Beta\bigg(\cos^2\bigg(\frac{\alpha_0}{2}\bigg);  \frac{d-2}{2},   \frac{d}{2}\bigg) & \bigg)^{-1} \times \\ \nonumber
& \bigg\{ 1 + \frac {q \, 2^{(d-2)/2}\, \Gamma((d-1)/2)} {\sqrt{\pi}\, \Gamma(d/2-1)} \,  \Beta\bigg(\cos^2\bigg(\frac{\alpha_0}{2}\bigg); \frac{d-2}{2}, \frac{1}{2}\bigg)  \bigg\}.
\end{empheq}

\qed

% ------------------------------------------------------------------------------- Theorem 2.6 -----------------------------------------------------------------------------------

\noindent{\bf Proof of Theorem \ref{msf-quadratic-ext-field-theo}.} Substituting expression $(\ref{quadratic-ext-field})$ into $(\ref{msfunc})$, we are presented with the following integral
\begin{align*}
& \int_\alpha^\pi  Q(\eta)  \, \bigg(\frac{1-\cos\alpha}{1-\cos\eta} \bigg)^{\frac{d-1}{2}} \bigg(\frac{1-\cos\alpha}{\cos\alpha-\cos\eta} \bigg)^{\frac{1}{2}}  {}_2 F_1\bigg (1,  \frac {d-1} {2};\frac {1} {2};   \frac{\cos\alpha-\cos\eta}{1-\cos\eta} \bigg) \, \sin^{d-2}\eta\, d\eta \\
& = (1-t)^{d/2} \, \int_{-1}^t  (1-u)^{-1}  \, (1+u)^{(d+1)/2}  \, (t-u)^{-1/2} \,  {}_2 F_1\bigg (1,  \frac {d-1} {2};\frac {1} {2};   \frac{t-u}{1-u} \bigg) \, du,
\end{align*}
where we set $t:=\cos\alpha$ and $u:=\cos\eta$. We thus need to evaluate
\begin{equation*}
J = \int_{-1}^t  (1-u)^{-1}  \, (1+u)^{(d+1)/2}  \, (t-u)^{-1/2} \,  {}_2 F_1\bigg (1,  \frac {d-1} {2};\frac {1} {2};   \frac{t-u}{1-u} \bigg) \, du.
\end{equation*}
This integral will be evaluated using the same approach we used when evaluating similar integral $(\ref{msfuncex_integI})$. Letting $a=-1$, $b=t$, $c=1$, $\gamma = 1/2$, $\beta = (d-1)/2$, $x = (b-a)/(c-a)$, and again using the substitution $(b-u)/(c-u) = x v$, we find
\begin{empheq}{align}\label{J-integral}
J & = \int_{-1}^t  (1-u)^{-1}  \, (1+u)^{(d+1)/2}  \, (t-u)^{-1/2} \,  {}_2 F_1\bigg (1,  \frac {d-1} {2};\frac {1} {2};   \frac{t-u}{1-u} \bigg) \, du \\ \nonumber
	& =  (c-b)^{\gamma-1} \, (b-a)^{\beta+1} \, x^\gamma \, \int_0^1 v^{\gamma-1} \, (1-v)^{\beta+1} \, (1-xv)^{-(\beta+\gamma+1)} \, {}_2 F_1 (1,\beta;\gamma;xv)\, dv \\ \nonumber
	& = (c-b)^{\gamma-1} \, (b-a)^{\beta+1} \, x^\gamma \, I,
\end{empheq}
where
\begin{equation*}
I:=\int_0^1 v^{\gamma-1} \, (1-v)^{\beta+1} \, (1-xv)^{-(\beta+\gamma+1)} \, {}_2 F_1 (1,\beta;\gamma;xv)\, dv.
\end{equation*}
Using the series representation for the Gauss hypergeometric function ${}_2 F_1$ \cite[p. 556, \# 15.1.1]{abram} and integrating term-by-term, we further write
\begin{empheq}{align}\label{msf-quadratic-ext-field-integ-3}
I & = \int_0^1 v^{\gamma-1} \, (1-v)^{\beta+1} \, (1-xv)^{-(\beta+\gamma+1)} \, {}_2 F_1 (1,\beta;\gamma;xv)\, dv \\ \nonumber
	& = \sum_{n=0}^\infty \frac {(\beta)_n} {(\gamma)_n} \, x^n \, \int_0^1 v^{n+\gamma-1} \, (1-v)^{\beta+1} \, (1-xv)^{-(\beta+\gamma+1)} \, dv \\ \nonumber
	& = \sum_{n=0}^\infty \frac {(\beta)_n} {(\gamma)_n} \, x^n \, I_n,
\end{empheq}
where
\begin{equation*}
I_n: = \int_0^1 v^{n+\gamma-1} \, (1-v)^{\beta+1} \, (1-xv)^{-(\beta+\gamma+1)} \, dv.
\end{equation*}
According to the integral representation of the function ${}_2 F_1$ \cite[p. 558, \# 15.3.1]{abram}, we further have
\begin{equation*}
I_n = \frac {\Gamma(n+\gamma) \, \Gamma(\beta+2)} {\Gamma(n+\gamma+\beta+2)} \, {}_2 F_1 (\beta+\gamma+1,n+\gamma;n+\gamma+\beta+2;x).
\end{equation*}
Recalling that the function ${}_2 F_1$ is symmetric with respect to switching the first two parameters \cite[p. 556, \# 15.1.1]{abram} , that is ${}_2 F_1 (\beta+\gamma+1,n+\gamma;n+\gamma+\beta+2;x) = {}_2 F_1 (n+\gamma, \beta+\gamma+1;n+\gamma+\beta+2;x)$, and using the fact that $n+\gamma+\beta +2 >  \beta+\gamma+1$ for all integer $n\geq 0$, we continue by using again the integral representation of ${}_2 F_1$ \cite[p. 558, \# 15.3.1]{abram} as follows,
\begin{align*}
I_n & = \frac {\Gamma(n+\gamma) \, \Gamma(\beta+2)} {\Gamma(n+\gamma+\beta+2)} \, {}_2 F_1 (\beta+\gamma+1,n+\gamma;n+\gamma+\beta+2;x) \\
	& = \frac {\Gamma(n+\gamma) \, \Gamma(\beta+2)} {\Gamma(n+\gamma+\beta+2)} \, {}_2 F_1 (n+\gamma, \beta+\gamma+1;n+\gamma+\beta+2;x) \\
	& = \frac {\Gamma(n+\gamma)\,\Gamma(\beta+2)} {\Gamma(\beta+\gamma+1)\, \Gamma(n+1)} \,  \int_0^1 v^{\beta+\gamma} \, (1-v)^n \, (1-xv)^{-(n+\gamma)} \, dv.
\end{align*}
Inserting the last expression into $(\ref{msf-quadratic-ext-field-integ-3})$, and switching the order of integration and summation, which is justified by the uniform convergence of the series as $v\in [0,1]$,  we have
\begin{align*}
I & = \sum_{n=0}^\infty \frac {(\beta)_n} {(\gamma)_n} \, x^n \, I_n \\
	& = \frac {\Gamma(\beta+2)} {\Gamma(\beta+\gamma+1)} \, \int_0^1 v^{\beta+\gamma} \, (1-xv)^{-\gamma} \, dv \, \sum_{n=0}^\infty \frac {(\beta)_n\, \Gamma(n+\gamma)} {\Gamma(n+1)\, (\gamma)_n} \, \bigg( x \, \frac {1-v}{1-xv} \bigg)^n
\end{align*}
It is not difficult to see that
\begin{equation*}
\sum_{n=0}^\infty \frac {(\beta)_n\, \Gamma(n+\gamma)} {\Gamma(n+1)\, (\gamma)_n} \, z^n  = \Gamma(\gamma) \, (1-z)^{-\beta}, \quad |z|<1.
\end{equation*}
We thus eventually find
\begin{align*}
I & = \frac {\Gamma(\beta+2) \, \Gamma(\gamma)} { \Gamma(\beta+\gamma+1)} \,  (1-x)^{-\beta} \,  \int_0^1 v^{\beta+\gamma} \, (1-xv)^{-(\gamma-\beta)} \, dv \\
	& = \frac {\sqrt{\pi}\, \Gamma(\beta+2)} {\Gamma(\beta+\gamma+2)} \, (1-x)^{-\beta} \, {}_2 F_1(\gamma-\beta,\beta+\gamma+1;\beta+\gamma+2;x),
\end{align*}
where we again used the integral representation of the hypergeometric function ${}_2 F_1$ \cite[p. 558, \# 15.3.1]{abram}. Inserting the latter into $(\ref{J-integral})$ and simplifying, we deduce that
\begin{align*}
& \int_\alpha^\pi  Q(\eta)  \, \bigg(\frac{1-\cos\alpha}{1-\cos\eta} \bigg)^{\frac{d-1}{2}} \bigg(\frac{1-\cos\alpha}{\cos\alpha-\cos\eta} \bigg)^{\frac{1}{2}}  {}_2 F_1\bigg (1,  \frac {d-1} {2};\frac {1} {2};   \frac{\cos\alpha-\cos\eta}{1-\cos\eta} \bigg) \, \sin^{d-2}\eta\, d\eta \\
& = \frac {2^{d/2-1}\, \sqrt{\pi}\, \Gamma((d+3)/2)} {\Gamma(d/2+2)} \, (1+\cos\alpha)^{d/2+1} \, {}_2 F_1\bigg (1-\frac{d}{2},\frac {d}{2}+1;\frac{d}{2}+2;\cos^2\bigg(\frac{\alpha}{2}\bigg) \bigg).
\end{align*}
The last expression can be simplified even further. Indeed, using a linear transformation formula for the hypergeometric function \cite[p. 559, \#15.3.3]{abram} and $(\ref{hypergeom-to-beta})$, we find obtain a neat formula
\begin{align*}
\int_\alpha^\pi  Q(\eta)  \, \bigg(\frac{1-\cos\alpha}{1-\cos\eta} \bigg)^{\frac{d-1}{2}} \bigg(\frac{1-\cos\alpha}{\cos\alpha-\cos\eta} \bigg)^{\frac{1}{2}} \,  {}_2 F_1\bigg (1, \frac {d-1} {2}; & \frac {1} {2}; \frac{\cos\alpha-\cos\eta}{1-\cos\eta} \bigg) \, \sin^{d-2}\eta\, d\eta \\
& = \frac {2^{d/2}\, \sqrt{\pi}\, \Gamma((d+3)/2)} {\Gamma(d/2+1)} \, \Beta \bigg( \cos^2\bigg(\frac {\alpha} {2} \bigg) ; \frac{d}{2}+1,\frac{d}{2}\bigg).
\end{align*}
Inserting the last integral into $(\ref{msfunc})$, we obtain the desired expression $(\ref{msf-quadratic-ext-field})$.

\qed

% ------------------------------------------------------------------------------- Theorem 2.7 -----------------------------------------------------------------------------------

\noindent{\bf Proof of Theorem \ref{density-rational-ext-field}.} Let 
\begin{empheq}{align}\label{func-w}
w(\alpha)  =  \frac {\sqrt{\pi}\,  \Gamma(d/2-1)} {2^{d-2}\,  \Gamma((d-1)/2)} \,   \bigg(\Beta\bigg(\cos^2\bigg(\frac {\alpha} {2} \bigg);  \frac{d-2}{2}, & \frac{d}{2} \bigg) \bigg)^{-1} \\ \nonumber 
& \bigg\{ 1 + \frac {2^{d/2}\, \Gamma((d+3)/2)} {\sqrt{\pi}\, \Gamma(d/2+1)} \, \Beta \bigg( \cos^2\bigg(\frac {\alpha} {2} \bigg) ; \frac{d}{2}+1,\frac{d}{2}\bigg)  \bigg\}.
\end{empheq}
Differentiating $w(\alpha)$, we find
\begin{equation}\label{func-w-der}
w'(\alpha) = \sin^{d-1}\bigg(\frac {\alpha} {2} \bigg) \, \cos^{d-3}\bigg(\frac {\alpha} {2} \bigg) \,  \bigg(\Beta\bigg(\cos^2\bigg(\frac {\alpha} {2} \bigg);  \frac{d-2}{2},  \frac{d}{2} \bigg) \bigg)^{-1} \, \omega(\alpha),
\end{equation}
where
\begin{equation}\label{func-omega}
\omega(\alpha) = w(\alpha) - \frac {4(d^2-1)}{d(d-2)} \,  \cos^4\bigg(\frac {\alpha} {2} \bigg).
\end{equation}
We therefore see that the critical points of $w(\alpha)$ are given by solutions of the equation
\begin{equation}\label{crit-points}
w(\alpha) = \frac {4(d^2-1)}{d(d-2)} \,  \cos^4\bigg(\frac {\alpha} {2} \bigg).
\end{equation}
Rearranging the latter, we obtain $(\ref{eqn-crit-pnts-quadratic-ext-field})$.

We continue by showing the existence and uniqueness of a critical point of $w(\alpha)$.  Our approach will be based on an argument developed in \cite{bds1}. First, observe that from  $(\ref{func-w})$ and $(\ref{func-omega})$ it follows that
\begin{equation*}
\lim_{\alpha \rightarrow \pi-} \omega (\alpha)  =  +\infty.
\end{equation*}
Hence, there is a smallest $\alpha_0\in[0,\pi)$ such that $\omega(\alpha)>0$ for $\alpha\in(\alpha_0,\pi)$. If $\alpha_0=0$, then $w(\alpha)$ is 
strictly increasing on $(0,\pi)$, and attains minimum at $\alpha=0$. If $\alpha_0>0$, we have that $\omega(\alpha)>0$ for $\alpha\in(\alpha_0,\pi)$. Taking into account the continuity of $\omega(\alpha)$, by passing to the limit $\alpha\rightarrow \alpha_0+$ in the latter inequality, we infer that $\omega(\alpha_0)\geq 0$. Since $\alpha_0$ was the smallest $\alpha$ such that $\omega(\alpha)>0$ on $(\alpha_0,\pi)$, we deduce that $\omega(\alpha_0) = 0$.

From expression $(\ref{func-w-der})$ it is clear that the sign of $w'(\alpha)$ is determined by the sign of $\omega(\alpha)$. This shows that $w'(\alpha)>0$ on $(\alpha_0,\pi)$, and $w'(\alpha_0)=0$. Next, suppose that $\xi\in(0,\pi)$ is a critical point of $w(\alpha)$, that is $w'(\xi)=0$. Using expression $(\ref{func-w-der})$, we readily find that
\begin{empheq}{align}\label{2ndder}
 w''(\alpha)  & = \bigg[ \sin^{d-1}\bigg(\frac {\alpha} {2} \bigg) \, \cos^{d-3}\bigg(\frac {\alpha} {2} \bigg) \,  \bigg(\Beta\bigg(\cos^2\bigg(\frac {\alpha} {2} \bigg);  \frac{d-2}{2},  \frac{d}{2} \bigg) \bigg)^{-1} \bigg]' \, \omega(\alpha)  \\ \nonumber 
	& + \sin^{d-1}\bigg(\frac {\alpha} {2} \bigg) \, \cos^{d-3}\bigg(\frac {\alpha} {2} \bigg) \,  \bigg(\Beta\bigg(\cos^2\bigg(\frac {\alpha} {2} \bigg);  \frac{d-2}{2},  \frac{d}{2} \bigg) \bigg)^{-1} \omega'(\alpha),
\end{empheq}
where
\begin{equation}\label{omega-der}
 \omega'(\alpha)  = w'(\alpha)  + \frac {8(d^2-1)}{d(d-2)} \cos^3\bigg(\frac {\alpha} {2} \bigg) \sin\bigg(\frac {\alpha} {2} \bigg).
\end{equation}
We want to show that $w''(\xi)>0$. This readily follows from $(\ref{2ndder})$, since for $0<\xi<\pi$,
\begin{equation*}
 w''(\xi)  =  \frac {8(d^2-1)}{d(d-2)} \cos^3\bigg(\frac {\xi} {2} \bigg)  \sin\bigg(\frac {\xi} {2} \bigg) > 0.
\end{equation*}

\noindent This means that $w(\alpha)$ has exactly one global minimum on $[0,\pi)$, which is either a unique solution $\alpha_0\in(0,\pi)$ of equation $(\ref{crit-points})$, if it exists, or $\alpha_0=0$, if such a solution does not exist.

We finish by computing the equilibrium density. Substituting expression $(\ref{quadratic-ext-field})$ into $(\ref{auxsSP})$, we are led to the following integral
\begin{equation*}
\int_\zeta^\pi \frac {Q(\theta)\, \cos^{d-3}(\theta/2)\, \sin\theta\, d\theta} {\sqrt{\cos\zeta-\cos\theta}} = \frac{1}{2^{(d-3)/2}} \, \int_\zeta^\pi (1+ \cos\theta)^{(d+1)/2} \, \frac {\sin\theta \, d\theta} {\sqrt{\cos\zeta-\cos\theta}}.
\end{equation*}
Making the change of variables $1+\cos\theta=(1+\cos\zeta)t$, we further write
\begin{align*}
 2^{-(d-3)/2} \, \int_\zeta^\pi  (1+ \cos\theta)^{(d+1)/2} \,   \frac {\sin\theta \, d\theta} {\sqrt{\cos\zeta-\cos\theta}} & = 2^{-(d-3)/2}\, (1+\cos\zeta)^{(d+2)/2} \, \int_0^1 t^{(d+1)/2} \, (1-t)^{-1/2} \, dt \\
& = 2^{-(d-3)/2}\, (1+\cos\zeta)^{(d+2)/2}\, \Beta\left(\frac{d+3}{2},\frac{1}{2}\right) \\
& = \frac {2^{5/2}\, \sqrt{\pi}\, \Gamma((d+3)/2)}{\Gamma(d/2+2)}\, \cos^{d+2} \bigg(\frac {\zeta} {2} \bigg).
\end{align*}
We thus conclude
\begin{equation}\label{quadratic-ext-field-func-g}
\int_\zeta^\pi \frac {Q(\theta)\, \cos^{d-3}(\theta/2)\, \sin\theta\, d\theta} {\sqrt{\cos\zeta-\cos\theta}} = \frac {2^{5/2}\, \sqrt{\pi}\, \Gamma((d+3)/2)}{\Gamma(d/2+2)}\, \cos^{d+2} \bigg(\frac {\zeta} {2} \bigg).
\end{equation}
Differentiating $(\ref{quadratic-ext-field-func-g})$ with respect to $\zeta$ and inserting the result into $(\ref{auxsSP})$, we find
\begin{equation}\label{quadratic-ext-field-func-g-expres}
g(\zeta) = - \frac {2^{5/2}\, \sqrt{\pi}\, \Gamma((d+3)/2)}{\Gamma(d/2+1)}\, \sin^{d-2} \bigg(\frac {\zeta} {2} \bigg)\, \cos^4 \bigg(\frac {\zeta} {2} \bigg), \quad \alpha_0 \leq \zeta \leq \pi.
\end{equation}
Substituting $(\ref{quadratic-ext-field-func-g-expres})$ into the right hand side of $(\ref{auxfSP})$, we arrive to the integral
\[
\int_{\alpha_0}^\eta  \sin^{d-2} \bigg(\frac {\zeta} {2} \bigg)\, \cos^4 \bigg(\frac {\zeta} {2} \bigg) \, \frac {\sin \zeta \, d\zeta} {\sqrt{\cos\zeta - \cos \eta}}.
\]
Making the change of variables $\zeta=\pi-y$, and setting $\widetilde{\alpha_0}=\pi-\alpha_0$, $\widetilde{\eta}=\pi-\eta$, we recast the latter integral as
\begin{align*}
\int_{\alpha_0}^\eta  \sin^{d-2} \bigg(\frac {\zeta} {2} \bigg)\, \cos^4 \bigg(\frac {\zeta} {2} \bigg) \, \frac {\sin \zeta \, d\zeta} {\sqrt{\cos\zeta - \cos \eta}} =  & \int_{\widetilde{\eta}}^{\widetilde{\alpha_0}}  \cos^{d-2} \bigg(\frac {y} {2} \bigg)\, \sin^4 \bigg(\frac {y} {2} \bigg) \, \frac {\sin y \, dy} {\sqrt{\cos y - \cos \widetilde{\eta}}} \\
 = 2^{-(d/2+1)}\, & \int_{\widetilde{\eta}}^{\widetilde{\alpha_0}} (1+ \cos y)^{(d-2)/2}\, (1-\cos y)^2 \, \frac {\sin y \, dy} {\sqrt{\cos y - \cos \widetilde{\eta}}}.
\end{align*}
Making a simple observation that $1-\cos y=(1+\cos y)-2$, allows us to continue the above string of integrals as
\begin{align*}
\int_{\widetilde{\eta}}^{\widetilde{\alpha_0}} (1+ \cos y)^{(d-2)/2}\, (1-\cos y)^2 \, \frac {\sin y \, dy} {\sqrt{\cos y - \cos \widetilde{\eta}}} & = \int_{\widetilde{\eta}}^{\widetilde{\alpha_0}} (1+ \cos y)^{(d-2)/2}\, ((1+\cos y)-2)^2 \, \frac {\sin y \, dy} {\sqrt{\cos y - \cos \widetilde{\eta}}} \\
& = \int_{\widetilde{\eta}}^{\widetilde{\alpha_0}} (1+ \cos y)^{(d-2)/2}\, (1+\cos y)^2 \, \frac {\sin y \, dy} {\sqrt{\cos y - \cos \widetilde{\eta}}} \\
& - 4\, \int_{\widetilde{\eta}}^{\widetilde{\alpha_0}} (1+ \cos y)^{(d-2)/2}\, (1+\cos y) \, \frac {\sin y \, dy} {\sqrt{\cos y - \cos \widetilde{\eta}}} \\
& + 4\, \int_{\widetilde{\eta}}^{\widetilde{\alpha_0}} (1+ \cos y)^{(d-2)/2}\,  \frac {\sin y \, dy} {\sqrt{\cos y - \cos \widetilde{\eta}}}.
\end{align*}
The three integrals on the right hand side of the last expression are evaluated using the change of variables $1+\cos y=(1+\cos\widetilde{\eta})t$. Performing these straightforward but tedious evaluations, and reverting back to $\alpha_0$ and $\eta$, we eventually obtain
\begin{align*}
\int_{\alpha_0}^\eta  \sin^{d-2} \bigg(\frac {\zeta} {2} \bigg)\, \cos^4 \bigg(\frac {\zeta} {2} \bigg) \,  \frac {\sin \zeta \, d\zeta} {\sqrt{\cos\zeta - \cos \eta}} &  \\ 
= 2^{-(d+2)/2}  \bigg\{  (1-\cos\eta)^{(d+3)/2}\, & \Beta \bigg( \frac{\cos\alpha_0-\cos\eta}{1-\cos\eta} ; \frac{1}{2},\frac{d}{2}+2\bigg) \\
 -  4\, (1-\cos\eta)^{(d+1)/2}\, & \Beta \bigg( \frac{\cos\alpha_0-\cos\eta}{1-\cos\eta} ; \frac{1}{2},\frac{d}{2}+1\bigg) \\
 + 4\, (1-\cos\eta)^{(d-1)/2}\,  & \Beta \bigg( \frac{\cos\alpha_0-\cos\eta}{1-\cos\eta} ; \frac{1}{2},\frac{d}{2}\bigg) \bigg\}.
\end{align*}
Differentiating the above expression with respect to $\eta$, and inserting the result into $(\ref{auxfSP})$, after simplifications we derive
\begin{empheq}{align}\label{F-quadratic-ext-field}
F(\eta) =  - \frac  { 2\Gamma((d+3)/2)} { d(d-2)\pi^{(d+1)/2} }  \, \bigg\{& \left( \frac {1-\cos\alpha_0} {1-\cos\eta} \right)^{d/2} \, \sqrt{\frac{1-\cos\eta}{\cos\alpha_0-\cos\eta}}\, (1+\cos\alpha_0)^2 \\ \nonumber
& + 2(d-1)\, \Beta \bigg( \frac{\cos\alpha_0-\cos\eta}{1-\cos\eta} ; \frac{1}{2},\frac{d}{2}\bigg) \\ \nonumber
& - 2(d+1)\,(1-\cos\eta)\,  \Beta \bigg( \frac{\cos\alpha_0-\cos\eta}{1-\cos\eta} ; \frac{1}{2},\frac{d}{2}+1 \bigg) \\ \nonumber
& + \frac{d+3}{2}\,(1-\cos\eta)^2\,  \Beta \bigg( \frac{\cos\alpha_0-\cos\eta}{1-\cos\eta} ; \frac{1}{2},\frac{d}{2}+2 \bigg) \bigg\}, \quad \alpha_0 \leq \eta \leq \pi,
\end{empheq}
The value of the Robin constant can now be found from $(\ref{constCQSP})$. However, going via this standard route with the function $F(\eta)$ of the type $(\ref{F-quadratic-ext-field})$ usually involves laborious calculations. Luckily, there is an alternative to that. Indeed, one observes that from the variational inequalities $(\ref{var1})$-$(\ref{var2})$ and Proposition \ref{msfmp} it follows that $\cF(C_{S,\alpha_0})=F_Q$. Therefore, using $(\ref{msf-quadratic-ext-field})$, we deduce that
\begin{equation*}
F_Q  =  \frac {\sqrt{\pi}\,  \Gamma(d/2-1)} {2^{d-2}\,  \Gamma((d-1)/2)} \,  \bigg(\Beta\bigg(\cos^2\bigg(\frac {\alpha_0} {2} \bigg);  \frac{d-2}{2},  \frac{d}{2} \bigg) \bigg)^{-1} \bigg\{1  + \frac {2^d\, \Gamma((d+3)/2)} {\sqrt{\pi} \, \Gamma(d/2+1)} \Beta \bigg( \cos^2\bigg(\frac {\alpha_0} {2} \bigg) ; \frac{d}{2}+1,\frac{d}{2}\bigg)\bigg\},
\end{equation*}
which in turn implies that
\begin{equation*}
C_Q  =  \frac {\Gamma(d/2-1)} {2^{d-1}\,  \pi^{d/2}} \, \bigg(\Beta\bigg(\cos^2\bigg(\frac {\alpha_0} {2} \bigg);  \frac{d-2}{2},  \frac{d}{2} \bigg) \bigg)^{-1} \bigg\{ 1  + \frac {2^d\, \Gamma((d+3)/2)} {\sqrt{\pi} \, \Gamma(d/2+1)} \Beta \bigg( \cos^2\bigg(\frac {\alpha_0} {2} \bigg) ; \frac{d}{2}+1,\frac{d}{2}\bigg)\bigg\}.
\end{equation*}
This completes the proof of the theorem.

\qed

\section{Acknowledgements} 
\noindent This work was done in partial fulfillment of Ph.D. degree at Oklahoma State University under the supervision of Prof. Igor E. Pritsker.

\end{document}